\documentclass[twocolumn,amsthm]{autart}

\usepackage[T1]{fontenc}
\usepackage{lmodern}
\usepackage{amsmath,amssymb,amsfonts,amsthm,mathtools}
\usepackage{mathrsfs}
\usepackage{bm}
\usepackage{enumitem}
\usepackage{graphicx}
\usepackage{etoolbox}
\patchcmd{\small}{\par}{}{}{}
\patchcmd{\footnotesize}{\par}{}{}{}
\usepackage{microtype}
\makeatletter
\@nameyeartrue
\makeatother
\let\AutartBibliography\thebibliography
\let\AutartEndBibliography\endthebibliography
\usepackage[authoryear,round]{natbib}
\let\thebibliography\AutartBibliography
\let\endthebibliography\AutartEndBibliography
\bibpunct{(}{)}{;}{a}{,}{,}
\AtBeginDocument{%
}
\AtBeginDocument{\urlstyle{same}}
\makeatletter
\let\cl@part\@empty
\makeatother
\usepackage{xcolor}
\usepackage{hyperref}
\AfterEndEnvironment{frontmatter}{\endNoHyper}
\newcommand{\doi}[1]{%
  \href{https://doi.org/#1}{doi:\nolinkurl{#1}}%
}

\hypersetup{
  colorlinks=true,
  linkcolor=blue,
  citecolor=blue,
  urlcolor=blue,
  pdfauthor={Xinyu Ma, Yang Liu, and Qingxin Meng},
  pdftitle={Recursive Stochastic Linear-Quadratic Control with Indefinite Weights for Jump-Diffusion Systems with Random Coefficients}
}

\theoremstyle{plain}
\newtheorem{theorem}{Theorem}
\newtheorem{lemma}{Lemma}
\newtheorem{proposition}{Proposition}
\newtheorem{corollary}{Corollary}
\theoremstyle{definition}
\newtheorem{definition}{Definition}
\newtheorem{assumption}{Assumption}
\newtheorem{example}{Example}
\theoremstyle{remark}
\newtheorem{remark}{Remark}

\DeclareRobustCommand{\namedref}[2]{\hyperref[#2]{#1~\ref*{#2}}}
\DeclareRobustCommand{\thmref}[1]{\namedref{Theorem}{#1}}
\DeclareRobustCommand{\lemref}[1]{\namedref{Lemma}{#1}}
\DeclareRobustCommand{\propref}[1]{\namedref{Proposition}{#1}}
\DeclareRobustCommand{\corref}[1]{\namedref{Corollary}{#1}}
\DeclareRobustCommand{\defref}[1]{\namedref{Definition}{#1}}
\DeclareRobustCommand{\assref}[1]{\namedref{Assumption}{#1}}
\DeclareRobustCommand{\remref}[1]{\namedref{Remark}{#1}}
\DeclareRobustCommand{\exref}[1]{\namedref{Example}{#1}}
\DeclareRobustCommand{\secref}[1]{\namedref{Section}{#1}}
\DeclareRobustCommand{\appref}[1]{\namedref{Appendix}{#1}}

\newcommand{\R}{\mathbb R}
\newcommand{\Pp}{\mathbb P}
\newcommand{\Qq}{\mathbb Q}
\newcommand{\F}{\mathbb F}
\newcommand{\E}{\mathbb E}
\newcommand{\Smat}{\mathbb S}
\newcommand{\Zmark}{\mathsf Z}
\newcommand{\cU}{\mathcal U}
\newcommand{\cX}{\mathcal X}

\newcommand{\cF}{\mathcal F}
\newcommand{\cP}{\mathcal P}
\newcommand{\ip}[2]{\left\langle #1,#2\right\rangle}
\newcommand{\tr}{\operatorname{tr}}
\newcommand{\essinf}{\operatorname*{ess\,inf}}

\newcommand{\Id}{I}
\newcommand{\dd}{\,\mathrm d}
\newcommand{\1}{\mathbf 1}
\newcommand{\transpose}{\mathsf T}

\makeatletter
\let\ps@copyright\ps@plain
\makeatother
\date{}

\begin{document}
\begin{frontmatter}

\runtitle{Recursive Stochastic Control with Jumps}

\title{Recursive Stochastic Linear--Quadratic Control with Indefinite Weights for Jump--Diffusion Systems with Random Coefficients\thanksref{footnoteinfo}}

\thanks[footnoteinfo]{This work was supported by the National Natural Science Foundation of China (No.~12671534; 12271158; 62673442), the National Key Research and Development Program of China (2024YFE0214000), the Natural Science Foundation of Zhejiang Province of China (Grant No.~LRG25F030002), and the Zhejiang Province Leading Geese Plan (Grant No.~2025C01056). Corresponding author: Qingxin Meng, tel.\ +86\,0572\,2321163, e-mail \texttt{mqx@huznu.edu.cn}.}

\author[ZJNU]{Xinyu Ma}\ead{mxyzfd@163.com},
\author[ZJNU,CMLAB,HZSA]{Yang Liu}\ead{liuyang@zjnu.edu.cn},
\author[HZNU]{Qingxin Meng}\ead{mqx@huznu.edu.cn}

\address[ZJNU]{School of Mathematical Sciences, Zhejiang Normal University, Jinhua 321004, China}
\address[CMLAB]{China--Mozambique Belt and Road Joint Laboratory on Smart Agriculture, Zhejiang Normal University, Jinhua 321004, China}
\address[HZSA]{Hangzhou School of Automation, Zhejiang Normal University, Hangzhou 311231, China}
\address[HZNU]{Department of Mathematical Sciences, Huzhou Normal University, Huzhou 313000, China}


\begin{abstract}
We study finite-horizon stochastic linear--quadratic (LQ) control for finite-activity jump diffusions with bounded random coefficients, indefinite weights and a linear recursive cost. The cost is defined by a backward stochastic differential equation (BSDE) whose generator has bounded deterministic coefficients multiplying the Brownian and jump integrands. A Brownian--Poisson change of measure and an integrating factor reduce the criterion to a nonrecursive LQ functional. On the intersection of the physical and transformed square-integrable control spaces, we establish recursive evaluation for $L^1$ quadratic data under the transformed measure and equality with the full transformed value. Under transformed uniform convexity and invertibility of the uncontrolled jump map, we identify the stochastic Riccati equation from the nonrecursive value kernel. Conditional energy estimates justify transferring its martingale coefficients to the physical measure. Completion of squares gives feedback representation, comparison and uniqueness in the feedback-admissible class. For each initial pair, the recursive value is attained exactly when the transformed optimizer has finite physical control energy. Analytical examples illustrate nonzero feedback and uniform convexity with a negative control weight.
\end{abstract}

\end{frontmatter}

\section{Introduction}\label{sec:intro}

Stochastic LQ control provides explicit links between random dynamics, quadratic costs and feedback laws. In financial and engineering models, uncertain coefficients and abrupt changes lead naturally to random-coefficient diffusions with jumps, while indefinite weights permit rewards as well as penalties. Recursive evaluation introduces a further dependence on future costs through a BSDE. This viewpoint is motivated by stochastic differential utility \citep{DuffieEpstein1992}, BSDE methods in finance \citep{ElKarouiPengQuenez1997} and generalized expectations \citep{Peng1997}. We study full-information control with homogeneous data, random coefficients, indefinite weights and a linear recursive cost, allowing the control to enter both diffusion and jump amplitudes.

For nonrecursive diffusion problems, \citet{Bismut1976,Tang2003,Tang2015} develop random-coefficient LQ and stochastic Riccati theory. Indefinite control weights are considered by \citet{ChenLiZhou1998}; solvability of indefinite stochastic Riccati equations is studied in \citet{HuZhou2003,Du2015}, and \citet{MoonDuncan2020} use completion of squares for random-coefficient indefinite LQ control. \citet{SunLiYong2016} distinguish open-loop and closed-loop solvability with deterministic coefficients. For random coefficients, \citet{SunXiongYong2021} establish Riccati solvability and closed-loop representations of open-loop optimizers under uniform convexity. Related nonrecursive results cover indefinite mean-field control \citep{LiLiYu2020} and robust nonhomogeneous systems \citep{WangXing2025}.

For BSDE-defined recursive costs, \citet{Peng1992DPP} establishes a dynamic programming principle, and \citet{Yong2008} studies LQ control under generalized expectation. \citet{ZhangZhang2021} apply Girsanov reduction to a deterministic-coefficient Brownian model. \citet{LiYong2026} formulate recursive LQ control at the natural $L^1$ level, characterize solvability through forward--backward equations and Riccati equations, and examine the control-domain issue in transferring a classical optimizer. Their principal matrices are deterministic and affine data are included. Their infinite-horizon work uses weighted $L^2$-stabilizability \citep{LiYongInfinite2026}. \citet{LiWangWu2020} treat recursive utility with delay under full and partial information.

A related dynamic-programming approach uses stochastic Hamilton--Jacobi--Bellman equations. \citet{Peng1992SHJB} studies these equations for stochastic control with random coefficients. In the recursive setting, \citet{MengDongShenTang2023} derive stochastic Hamilton--Jacobi--Bellman equations for random-coefficient jump diffusions, and \citet{DongMengZhang2024} investigate the relation between  maximum principle and dynamic programming in the Brownian case. The present LQ analysis concerns matrix Riccati equations and the integrability of their feedback controls, rather than regularity of a general nonlinear value field.

A closely related work is \citet{MoonBasar2022}, who developed a dynamic programming framework for recursive stochastic control with jump diffusions and random coefficients, and applied it to a recursive indefinite LQ problem. For the LQ application, the associated integro-type stochastic Riccati differential equation (ISRDE) and a state-feedback characterization were
derived, while the verification result assumes the existence of an ISRDE solution and the uniform positive definiteness of the effective control weight. The solvability of the ISRDE was explicitly left as an open problem \citep[Remark~6(2)]{MoonBasar2022}. Motivated by this issue, we establish solvability conditions for the associated recursive Riccati equation via an equivalent nonrecursive LQ formulation and verify the admissibility of the resulting feedback control under the original
probability measure.

For nonrecursive jump LQ problems, \citet{ZhangDongMeng2020} establish Riccati theory with nonnegative state and terminal weights and a uniformly positive definite control weight. \citet{LiWuYu2018} handle indefinite weights by a relax compensator. Stochastic flows and jump maximum-principle methods are developed in \citet{Kunita1990,TangLi1994}. We use the nonrecursive value representation and semimartingale construction of \citet{MaMeng2026} as the starting point for Riccati identification.

The recursive jump problem links three analytical issues. Square-integrable states and controls generate $L^1$ quadratic data; their evaluation can require a locally absolutely integrable jump integrand rather than a locally square-integrable one. Changing measure preserves null sets but can change the square-integrable control space, so value reduction and optimizer transfer must be treated separately. In addition, the effective control weight contains the post-jump value kernel, and the optimal jump map may be singular even when the uncontrolled jump map is invertible.

We resolve these questions on a common control domain. First, we construct the $L^1$ recursive evaluation and use bounded-control approximation to prove equality with the full transformed value. Second, under transformed uniform convexity and uncontrolled jump-map invertibility, auxiliary-cost comparison and optimal trajectories restarted between jumps identify the Riccati equation. Conditional energy estimates give the integrated moments needed for its physical-measure representation. Third, completion of squares yields feedback representation, comparison and uniqueness in the feedback-admissible class. For each initial pair, physical square integrability of the transformed optimizer is the exact attainment condition; bounded gain provides a sufficient case. The analysis assumes finite activity and bounded deterministic coefficients on the generator's Brownian and jump integrands. Portfolio benchmarks and a negative-control-weight example illustrate the results.

\namedref{Sections}{sec:setting}--\ref{sec:recursive} establish the model, evaluation and value reduction; \namedref{Sections}{sec:transformed}--\ref{sec:feedback} give the Riccati and verification results. \secref{sec:finance} presents the examples, and \appref{app:direct-check} derives physical-measure first-order conditions.

\section{Probability measures, spaces, and state dynamics}\label{sec:setting}

\subsection{The physical model}

Let $(\Omega,\cF,\F,\Pp)$ be a complete filtered probability space on a finite horizon $[0,T]$ with $0<T<\infty$.  The filtration $\F=(\cF_t)_{0\le t\le T}$ is the usual augmentation of the filtration generated jointly by a one-dimensional $\Pp$-Brownian motion $W$ and a Poisson random measure $\mu$ on $[0,T]\times\Zmark$.  Its compensator is $\nu(de)\dd t$, where $\Zmark$ is a standard Borel mark space and $\nu$ is a finite measure, and $W$ and $\mu$ are independent.  We write
$
\widetilde\mu(\dd t,\dd e)=\mu(\dd t,\dd e)-\nu(\dd e)\dd t.
$
Conditional on the total Poisson count, the ordered atom times are order statistics of independent uniform variables on $[0,T]$. Thus, almost surely, no two atoms share a time coordinate and no atom occurs at  $T$:
\begin{equation}\label{eq:time-simple}
\mu(\{s\}\times\Zmark)\le1\ (0\le s\le T),\qquad
\mu(\{T\}\times\Zmark)=0.
\end{equation}
A predictable product-null set is avoided by all atoms, by compensation. These assertions also hold under the equivalent measure below.

Write $\Smat^n$ for symmetric matrices, $|\cdot|$ for Euclidean or Frobenius norms and $\|\cdot\|$ for the operator norm. For symmetric $H,K$, $H\succeq K$ denotes the Loewner order, and $H\succ K$ strict positive definiteness of $H-K$. We use almost surely (a.s.), almost everywhere (a.e.), and $\E_t^{\mathbb M}[\cdot]=\E^{\mathbb M}[\cdot\mid\cF_t]$. For a Euclidean space $H$, $p\ge1$ and an equivalent probability $\mathbb M$, let $L^p_{\mathbb M,\cF_t}(H)$ be the $\cF_t$-measurable $H$-valued variables with finite $p$th moment, and $L^\infty_{\cF_t}(H)$ the essentially bounded ones. The spaces $\cU^p_{\mathbb M}[t,T]$ and $\mathcal S^p_{\mathbb M}[t,T]$ consist respectively of predictable $\R^m$-valued controls and adapted c\`adl\`ag (right-continuous with left limits) $\R^n$-valued states satisfying
\[
\E^{\mathbb M}\Big(\int_t^T|u_s|^2\dd s\Big)^{p/2}<\infty,
\qquad
\E^{\mathbb M}\sup_{t\le s\le T}|X_s|^p<\infty.
\]
We omit the exponent when $p=2$.

For $(t,\xi)\in[0,T]\times L^2_{\Pp,\cF_t}(\R^n)$ and $u\in\cU_{\Pp}[t,T]$, consider on $[t,T]$ the stochastic differential equation (SDE)
\begin{equation}\label{eq:state-P}
\left\{\begin{aligned}
\dd X_s={}&(A_s X_{s-}+B_s u_s)\dd s+(C_sX_{s-}+D_su_s)\dd W_s\\
&+\int_{\Zmark}[E(s,e)X_{s-}+F(s,e)u_s]\widetilde\mu(\dd s,\dd e),\\
X_t={}&\xi.
\end{aligned}\right.
\end{equation}
Here $A,C$ take values in $\R^{n\times n}$, $B,D$ in $\R^{n\times m}$, $E$ in $\R^{n\times n}$, and $F$ in $\R^{n\times m}$. Left limits specify predictable stochastic integrands; current and left-limit values agree in time integrals. Initial optimization times $t$ are deterministic. In stopped equations, integrals after a stopping time $\tau$ run over $(\tau,s]$, with an observed jump at $\tau$ included in the initial value.

\begin{assumption}\label{ass:state}
The processes $A,B,C,D$ are bounded and predictable.  The fields $E,F$ are bounded and $\cP\otimes\mathscr B(\Zmark)$-measurable, where $\cP$ is the predictable $\sigma$-field.  Moreover, $\nu(\Zmark)<\infty$.
\end{assumption}

Under \assref{ass:state}, \eqref{eq:state-P} has a unique c\`adl\`ag strong solution. For $p=2$, its standard estimate is
\begin{equation}\label{eq:state-estimate-P}
\E^{\Pp}\sup_{t\le s\le T}|X(s)|^2
\le C\left(\E^{\Pp}|\xi|^2+\E^{\Pp}\int_t^T|u(s)|^2\dd s\right).
\end{equation}
The conditional and difference estimates hold as well. For $p>2$, the appropriate estimate is
\begin{equation}\label{eq:state-estimate-higher}
\begin{aligned}
\E^{\Pp}\sup_{t\le s\le T}|X(s)|^p
\le &C_p\Bigg[\E^{\Pp}|\xi|^p+
\E^{\Pp}\Big(\int_t^T|u_s|^2\dd s\Big)^{p/2}\\
&+\E^{\Pp}\int_t^T\int_{\Zmark}|F(s,e)u_s|^p\nu(\dd e)\dd s\Bigg].
\end{aligned}
\end{equation}
These estimates follow from linear SDE bounds and maximal inequalities for compensated Poisson integrals \citep{Situ2005,MarinelliRoeckner2014}. The constants depend on the dimensions, horizon, coefficient bounds and compensator mass. The $p>2$ estimate requires finiteness of its full right-hand side, including the jump-control term; the $p=2$ estimate applies to every admissible control.

Let $Q:[0,T]\times\Omega\to\Smat^n$, $R:[0,T]\times\Omega\to\Smat^m$, and $S:[0,T]\times\Omega\to\R^{m\times n}$ be predictable, and let $G:\Omega\to\Smat^n$ be $\cF_T$-measurable.  Define
\begin{equation}\label{eq:running-cost}
\ell(s,x,u):=\ip{Q(s)x}{x}+2\ip{S(s)x}{u}+\ip{R(s)u}{u}.
\end{equation}

\begin{assumption}\label{ass:weights}
The processes $Q,R,S$ and the random matrix $G$ are essentially bounded. No positivity condition is imposed on $Q,R,G$.
\end{assumption}

\subsection{The recursive coefficients and the equivalent measure}

The recursive generator will contain a utility-rate process $a$, a Brownian aggregator $\beta$, and a jump aggregator $\eta$.

\begin{assumption}\label{ass:recursive-coeff}
The process $a$ is bounded and predictable.  The function $\beta:[0,T]\to\R$ is deterministic, Borel measurable and bounded.  The function $\eta:\Zmark\to\R$ is deterministic, Borel measurable and bounded, and there exists $\varepsilon_\eta>0$ such that
$1+\eta(e)\ge\varepsilon_\eta$ for $\nu$-a.e. $e\in\Zmark$.
\end{assumption}

Define the Dol\'eans exponential density by
\begin{equation}\label{eq:density-explicit}
\begin{aligned}
\mathcal E(t)&=\exp\Bigg\{\int_0^t\beta(s)\dd W(s)-\frac12\int_0^t\beta(s)^2\dd s\\
&+\int_0^t\int_{\Zmark}\log(1+\eta(e))\mu(\dd s,\dd e)-t\int_{\Zmark}\eta(e)\nu(\dd e)\Bigg\}.
\end{aligned}
\end{equation}

\begin{lemma}\label{lem:density-moments}
Under \assref{ass:recursive-coeff}, $\mathcal E$ is a strictly positive uniformly integrable $\Pp$-martingale.  For every $q\in\R$,
\begin{equation}\label{eq:density-moment-formula}
\begin{aligned}
\E^{\Pp}[\mathcal E(T)^q]&=\exp\Bigg\{\frac{q(q-1)}2\int_0^T\beta(s)^2\dd s\\
&+T\int_{\Zmark}\big[(1+\eta(e))^q-1-q\eta(e)\big]\nu(\dd e)\Bigg\}<\infty.
\end{aligned}
\end{equation}
Consequently, if $\Qq$ is defined by
$(\dd\Qq/\dd\Pp)|_{\cF_T}=\mathcal E(T)$, then $\Pp$ and $\Qq$ are equivalent and $\mathcal E(T)^{-1}$ has moments of every order under $\Qq$.
\end{lemma}

\begin{proof}
The Brownian and Poisson factors in \eqref{eq:density-explicit} are independent. The Gaussian moment formula and the exponential formula for a Poisson random measure give \eqref{eq:density-moment-formula}. For every real $q$, its integrals are finite because $\beta,\eta$ are bounded and $1+\eta\ge\varepsilon_\eta>0$. Independent increments give the martingale property, and $q=1$ gives unit expectation at $T$; hence the martingale is uniformly integrable. Its paths are strictly positive. Finally,
$\E^{\Qq}[\mathcal E_T^{-q}]=\E^{\Pp}[\mathcal E_T^{1-q}]<\infty$.
\end{proof}

By the Girsanov theorem for Brownian motion and Poisson random measures (see, e.g., \citet{Protter2005}),
\begin{equation}\label{eq:Q-noises}
\begin{aligned}
W^{\circ}(t)&:=W(t)-\int_0^t\beta(s)\dd s,\\
\widetilde\mu^{\circ}(\dd t,\dd e)&:=\widetilde\mu(\dd t,\dd e)-\eta(e)\nu(\dd e)\dd t
\end{aligned}
\end{equation}
are, respectively, a $\Qq$-Brownian motion and the compensated version of $\mu$ under $\Qq$.  The new compensator is
$\nu^{\circ}(\dd e)\dd t:=(1+\eta(e))\nu(\dd e)\dd t$.
The measure $\nu^{\circ}$ is finite and equivalent to $\nu$, with density bounded above and away from zero. Deterministic $\beta$ preserves the Brownian path information, while deterministic $\eta$ gives a deterministic Poisson compensator. Factorization of the density preserves independence under $\Qq$, giving the Brownian--Poisson representation used below. The random utility rate $a$ enters only the bounded finite-variation factor \eqref{eq:kappa}.

Define the transformed drift coefficients
\begin{align}
A^{\circ}(t)&:=A(t)+\beta(t)C(t)
+\int_{\Zmark}\eta(e)E(t,e)\nu(\dd e),\notag\\
B^{\circ}(t)&:=B(t)+\beta(t)D(t)
+\int_{\Zmark}\eta(e)F(t,e)\nu(\dd e).
\label{eq:transformed-drift}
\end{align}
Substitution of \eqref{eq:Q-noises} into \eqref{eq:state-P} gives the $\Qq$-equation
\begin{equation}\label{eq:state-Q}
\left\{\begin{aligned}
\dd X_s={}&(A_s^\circ X_{s-}+B_s^\circ u_s)\dd s+(C_sX_{s-}+D_su_s)\dd W_s^\circ\\
&+\int_{\Zmark}[E(s,e)X_{s-}+F(s,e)u_s]\widetilde\mu^\circ(\dd s,\dd e),\\
X_t={}&\xi.
\end{aligned}\right.
\end{equation}
For a control and initial state that are square integrable under both measures, equations \eqref{eq:state-P} and \eqref{eq:state-Q} describe the same pathwise state process.  For a general $u\in\cU_{\Qq}[t,T]$, equation \eqref{eq:state-Q} is well posed in $L^2(\Qq)$.

\section{The recursive criterion and exact LQ reduction}\label{sec:recursive}

\subsection{The common control domain}

Set
\begin{equation}\label{eq:common-spaces}
\begin{aligned}
\cX_t^2&:=L^2_{\Pp,\cF_t}(\R^n)\cap L^2_{\Qq,\cF_t}(\R^n),\\
\cU_{\Pp,\Qq}[t,T]&:=\cU_{\Pp}[t,T]\cap\cU_{\Qq}[t,T].
\end{aligned}
\end{equation}
Equip each intersection with the Hilbert norm whose squared value is the sum of the two squared $L^2$ norms. The $\Pp$ requirement preserves the physical control energy, while the $\Qq$ requirement ensures integrability of the quadratic evaluation data. The graph spaces are Hilbert: the two $L^2$ limits agree by equivalence of the measures. They contain all bounded initial states and bounded predictable controls. Let
\begin{equation}\label{eq:kappa}
\kappa(t):=e^{\int_0^t a(s)\dd s},\qquad
0<\underline\kappa\le\kappa(t)\le\overline\kappa,
\end{equation}
where $\underline\kappa=e^{-\|a\|_\infty T}$ and $\overline\kappa=e^{\|a\|_\infty T}$ are deterministic bounds holding simultaneously on $[0,T]$, $\Pp$-a.s.

For $(t,\xi)\in[0,T]\times\cX_t^2$ and $u\in\cU_{\Pp,\Qq}[t,T]$, let $X$ be the corresponding state and consider
\begin{equation}\label{eq:recursive-BSDE}
\left\{\begin{aligned}
\dd Y_s={}&-\Big[a_sY_s+\beta_sZ_s+\ell(s,X_s,u_s)\\
&\qquad+\int_{\Zmark}\eta(e)U(s,e)\nu(\dd e)\Big]\dd s\\
&+Z_s\dd W_s+\int_{\Zmark}U(s,e)\widetilde\mu(\dd s,\dd e),\\
Y_T={}&\ip{GX_T}{X_T}.
\end{aligned}\right.
\end{equation}
The terminal quadratic term and the integrated running cost belong to $L^1(\Qq)$, so we use the following local integration spaces.

All local integration spaces refer to the closed interval $[t,T]$. For $\mathbb M\in\{\Pp,\Qq\}$ and a finite mark measure $m$, the spaces $\mathcal H^2_{\rm loc}(\mathbb M)$ and $\mathcal L^1_{\rm loc}(m,\mathbb M)$ consist of predictable $Z$ and $U$, respectively, whose quadratic-energy and absolute-compensator integrals are pathwise finite at $T$ and satisfy
$
\E^{\mathbb M}\int_t^{\rho_j}|Z_s|^2\dd s<\infty$, $
\E^{\mathbb M}\int_t^{\rho_j}\int_{\Zmark}|U(s,e)|m(\dd e)\dd s<\infty
$
for increasing stopping times $\rho_j\le T$ that eventually equal $T$ almost surely. Localizers for the two integrands may be combined by taking their minimum. The jump integral is locally the atom sum minus its compensator; no compensator-based local $L^2$ condition is imposed on $U$.

\begin{lemma}\label{lem:local-measures}
Under \assref{ass:recursive-coeff}, the spaces satisfy
\begin{equation}\label{eq:local-class-equality}
\mathcal H^2_{\rm loc}(\Pp)=\mathcal H^2_{\rm loc}(\Qq),\quad\mathcal L^1_{\rm loc}(\nu,\Pp)=\mathcal L^1_{\rm loc}(\nu^\circ,\Qq).
\end{equation}
The noise substitutions \eqref{eq:Q-noises} hold for their stochastic integrals under both measures.
\end{lemma}
\begin{proof}
Pathwise finiteness is preserved by equivalent probabilities and by $c_-\nu\le\nu^\circ\le c_+\nu$. Stop the continuous adapted clock
\[
H_s=\int_t^s\Big[|Z_r|^2+\int_{\Zmark}|U(r,e)|(\nu+\nu^\circ)(\dd e)\Big]\dd r
\]
on reaching $j$, capped at $T$. The stopped clock is at most $j$ under either measure and the stops eventually equal $T$. Predictability of $\1_{(t,\rho_j]}$ and compensation give local absolute integrability of the atom sums. Subtracting compensators proves the jump substitution; local approximation proves the Brownian one. Their drift corrections are finite by Cauchy--Schwarz and $H_T<\infty$. Countably many localizers give indistinguishable identities on $[t,T]$.
\end{proof}

\begin{lemma}\label{lem:L1-representation}
Let $\zeta\in L^1_{\Qq,\cF_T}(\R)$ and $M(s)=\E^{\Qq}[\zeta\mid\cF_s]$. There are predictable $Z\in\mathcal H^2_{\mathrm{loc}}(\Qq)$ and $U\in\mathcal L^1_{\mathrm{loc}}(\nu^{\circ},\Qq)$ such that
\begin{equation}\label{eq:L1-martingale-representation}
M_s=M_t+\int_t^sZ_r\dd W_r^\circ+
\int_t^s\int_{\Zmark}U(r,e)\widetilde\mu^\circ(\dd r,\dd e).
\end{equation}
They are unique up to their corresponding predictable null sets.
\end{lemma}

\begin{proof}
Since $\zeta\in L^1(\Qq)$, $M_s=\E_s^{\Qq}\zeta$ has a uniformly integrable martingale version on $[0,T]$. Choose a Borel isomorphism $b:\Zmark\to B\subset(1,2)$. The pair consisting of $W^\circ$ and $\int_0^\cdot\int b(e)\mu(\dd s,\dd e)$ has independent increments; its nonzero jump sizes recover all marks, so it generates the same augmented filtration. Keeping the process and filtration constant after $T$, and adjoining the completion null sets at time zero, gives condition~III.2.12 of \citet{JacodShiryaev2003} on $\cF_T$. Their Theorem~III.4.34(a), p.~189, supplies the Brownian integrand and a general compensated-random-measure integrand. Restrict the representation to $[t,T]$. Since the compensator has no time atoms, its jumps satisfy
\[
\Delta M_s=\int_{\Zmark}U(s,e)\mu(\{s\},\dd e).
\]
By \eqref{eq:time-simple}, this is $U(\sigma_j,e_j)$ at the $j$th atom after $t$, and there are no other jumps. Set $\rho_k=T\wedge\sigma_k$, with $\sigma_k=\infty$ when necessary, and $\lambda_\circ=\nu^\circ(\Zmark)$. Optional sampling and $\E^{\Qq}|M_{s-}|\le\E^{\Qq}|\zeta|$ give
\[
\begin{aligned}
\E^{\Qq}\sum_{t<s\le\rho_k}|\Delta M_s|
&\le\sum_{j=1}^k\E^{\Qq}[\1_{\{\sigma_j\le T\}}|M_{\sigma_j\wedge T}|]\\
&\quad+\lambda_\circ\E^{\Qq}\int_t^T|M_{s-}|\dd s\\
&\le[k+\lambda_\circ(T-t)]\E^{\Qq}|\zeta|.
\end{aligned}
\]
For the pre-jump term we used compensation; the bound on $M_{s-}$ follows by Fatou along deterministic times. Apply compensation first to the bounded nonnegative predictable field $\1_{\{t<s\le\rho_k\}}(|U|\wedge N)$, and then let $N\uparrow\infty$. The jump identity and simplicity imply
\[
\E^{\Qq}\int_t^{\rho_k}\int_{\Zmark}|U|\nu^\circ(\dd e)\dd s
=\E^{\Qq}\sum_{t<s\le\rho_k}|\Delta M_s|<\infty.
\]
Proposition~II.1.28 of \citet{JacodShiryaev2003} identifies each stopped jump integral with the atom sum minus its compensator. Finite activity makes $\rho_k=T$ eventually almost surely. Combining these stops with continuous-energy localizers proves the claimed representation on the closed horizon. Continuous quadratic variation identifies $Z$ uniquely, and compensation of the indicator of nonzero jump-integrand differences identifies $U$ in product measure.
\end{proof}

\begin{definition}\label{def:recursive-solution}
A triple $(Y,Z,U)$ is an \emph{integrable recursive solution} of \eqref{eq:recursive-BSDE} if $Y$ is c\`adl\`ag and adapted, $Z,U$ belong to the common local spaces of \lemref{lem:local-measures}, the physical-noise integral equation \eqref{eq:recursive-BSDE} holds indistinguishably on $[t,T]$, and
\begin{equation}\label{eq:UI-class}
\begin{aligned}
\Big\{\kappa(\tau)Y(\tau)+\int_t^\tau\kappa(s)\ell(s,X(s),u(s))\dd s:{}&\\
\tau\text{ is an }\F\text{-stopping time in }[t,T]&\Big\}
\end{aligned}
\end{equation}
forms a uniformly integrable family under $\Qq$. The uniformly integrable family is taken under $\Qq$ (class $(D)$).
\end{definition}

\begin{theorem}\label{thm:recursive-representation}
Let \namedref{Assumptions}{ass:state}, \ref{ass:weights} and \ref{ass:recursive-coeff} hold.  For every $(t,\xi)\in[0,T]\times\cX_t^2$ and $u\in\cU_{\Pp,\Qq}[t,T]$, equation \eqref{eq:recursive-BSDE} has a unique integrable recursive solution.  Moreover,
\begin{equation}\label{eq:recursive-representation}
\begin{aligned}
Y_s=\kappa_s^{-1}\E_s^{\Qq}\Big[&\kappa_T\ip{GX_T}{X_T}\\
&+\int_s^T\kappa_r\ell(r,X_r,u_r)\dd r\Big],\qquad s\in[t,T].
\end{aligned}
\end{equation}
In particular, the recursive cost
$J^{\mathrm r}(t,\xi;u):=Y(t)$
belongs to $L^1_{\Qq,\cF_t}(\R)$ and is given by \eqref{eq:recursive-representation} at $s=t$.
\end{theorem}

\begin{proof}
Under $\Qq$, the state estimate associated with \eqref{eq:state-Q} and the boundedness of the weights imply
\[
\begin{aligned}
\zeta:={}&\kappa(T)\ip{GX(T)}{X(T)}\\
&+\int_t^T\kappa(r)\ell(r,X(r),u(r))\dd r\in L^1_{\Qq,\cF_T}(\R).
\end{aligned}
\]
The uniformly integrable martingale $M_s=\E_s^{\Qq}\zeta$ has representation \eqref{eq:L1-martingale-representation}, with integrands $\bar Z\in\mathcal H^2_{\mathrm{loc}}(\Qq)$ and $\bar U\in\mathcal L^1_{\mathrm{loc}}(\nu^\circ,\Qq)$. Define
\begin{align*}
Y_s&:=\kappa_s^{-1}\Big(M_s-\int_t^s\kappa_r\ell(r,X_r,u_r)\dd r\Big),\\
Z_s&:=\kappa_s^{-1}\bar Z_s,\qquad U(s,e):=\kappa_s^{-1}\bar U(s,e).
\end{align*}
Since $\dd\kappa=a\kappa\dd s$, the product rule and the noise substitutions in \lemref{lem:local-measures} give \eqref{eq:recursive-BSDE}. Its terminal condition follows from $M_T=\zeta$. The process in \eqref{eq:UI-class} equals $M$, so it is uniformly integrable. Conditional expectation yields \eqref{eq:recursive-representation}; conditional Jensen gives $\E^{\Qq}|J^{\mathrm r}(t,\xi;u)|\le\underline\kappa^{-1}\E^{\Qq}|\zeta|<\infty$.

For uniqueness, let another triple satisfy \defref{def:recursive-solution}.  Reversing the preceding calculation shows that the process in \eqref{eq:UI-class} is a uniformly integrable $\Qq$-martingale with terminal value $\zeta$.  It must therefore equal $M$.  The predictable representation property then identifies $Z$ and $U$ uniquely up to the usual null sets.
\end{proof}

\begin{remark}\label{rmk:L1-counterexample}
Let $N$ be a unit-rate Poisson process on $[0,1]$, $u(s)=s^{-1/3}$ for $s>0$, and
$X(s)=\int_0^su(r)(\dd N(r)-\dd r)$. Then $u\in L^2(0,1)$ and $X(1)^2\in L^1$. For zero recursive coefficients and zero running cost,
\[
\begin{aligned}
Y(s)&=\E[X(1)^2\mid\cF_s]=X(s)^2+\int_s^1u(r)^2\dd r,\\
U(s)&=2X(s-)u(s)+u(s)^2.
\end{aligned}
\]
Before the first jump, $X(s-)=-\tfrac32s^{2/3}$, hence
$U(s)=s^{-2/3}-3s^{1/3}$. Its $|U|$ compensator integral is finite near zero, but its $|U|^2$ integral is infinite on every initial interval. Thus the local $L^1$ jump class is essential even for this quadratic evaluation.
\end{remark}

\subsection{The transformed LQ problem and equality of value flows}

Define the transformed weights
\begin{equation}\label{eq:transformed-weights}
(\widehat Q,\widehat R,\widehat S)=\kappa(Q,R,S),\qquad\widehat G=\kappa(T)G.
\end{equation}
Set
$\widehat\ell(t,x,u):=\kappa(t)\ell(t,x,u)$.
For $(t,\xi)\in[0,T]\times L^2_{\Qq,\cF_t}(\R^n)$ and $u\in\cU_{\Qq}[t,T]$, let $X^{\circ}$ solve \eqref{eq:state-Q} and set
\begin{equation}\label{eq:J-hat}
\widehat J(t,\xi;u)=\E_t^{\Qq}\Big[\ip{\widehat G X_T^\circ}{X_T^\circ}
+\int_t^T\widehat\ell(s,X_s^\circ,u_s)\dd s\Big].
\end{equation}
By \thmref{thm:recursive-representation},
\begin{equation}\label{eq:J-scaling}
J^{\mathrm r}(t,\xi;u)=\kappa(t)^{-1}\widehat J(t,\xi;u),
\qquad u\in\cU_{\Pp,\Qq}[t,T].
\end{equation}

The common control space need not equal either full $L^2$ space, but it is dense in both.

\begin{lemma}\label{lem:space-transfer}
For every $t\in[0,T]$, bounded predictable controls are contained in $\cU_{\Pp,\Qq}[t,T]$ and are dense in $\cU_{\Qq}[t,T]$ and in $\cU_{\Pp}[t,T]$.  In addition, for every $p>1$,
\begin{equation}\label{eq:moment-transfer}
\cU^{2p}_{\Qq}[t,T]\subset\cU_{\Pp}[t,T],
\qquad
\cU^{2p}_{\Pp}[t,T]\subset\cU_{\Qq}[t,T].
\end{equation}
The analogous inclusions hold for terminal random variables.
\end{lemma}

\begin{proof}
The bounded predictable truncations $u^N=u\1_{\{|u|\le N\}}$ belong to both spaces and converge in whichever $L^2$ space contains $u$.

For the first inclusion in \eqref{eq:moment-transfer}, put $H=\int_t^T|u(s)|^2\dd s$ and let $q=p/(p-1)$.  By H\"older's inequality and \lemref{lem:density-moments},
\[
\begin{aligned}
\E^{\Pp}H&=\E^{\Qq}[\mathcal E(T)^{-1}H]\\
&\le\bigl(\E^{\Qq}\mathcal E(T)^{-q}\bigr)^{1/q}\bigl(\E^{\Qq}H^p\bigr)^{1/p}<\infty.
\end{aligned}
\]
The other inclusion and the statement for random variables follow in the same way.
\end{proof}

The essential infima below are taken in the extended-real $\cF_t$-measurable random variables modulo almost-sure equality, common to $\Pp$ and $\Qq$. Define
\begin{equation}\label{eq:value-definitions}
\begin{aligned}
V^{\mathrm r}(t,\xi)&:=\essinf_{u\in\cU_{\Pp,\Qq}[t,T]}J^{\mathrm r}(t,\xi;u),\\
\widehat V(t,\xi)&:=\essinf_{u\in\cU_{\Qq}[t,T]}\widehat J(t,\xi;u).
\end{aligned}
\end{equation}

\begin{proposition}\label{prop:value-equality}
Under \namedref{Assumptions}{ass:state}, \ref{ass:weights} and \ref{ass:recursive-coeff}, for every $(t,\xi)\in[0,T]\times\cX_t^2$,
\begin{equation}\label{eq:value-equality}
\essinf_{u\in\cU_{\Pp,\Qq}[t,T]}\widehat J(t,\xi;u)
=\essinf_{u\in\cU_{\Qq}[t,T]}\widehat J(t,\xi;u),
\qquad\Qq\text{-a.s.}
\end{equation}
The equality is understood in the extended real sense when the problem is not bounded below. Consequently,
\begin{equation}\label{eq:value-scaling}
V^{\mathrm r}(t,\xi)=\kappa(t)^{-1}\widehat V(t,\xi).
\end{equation}
\end{proposition}

\begin{proof}
For fixed $u\in\cU_{\Qq}[t,T]$, its bounded truncations converge in $\cU_{\Qq}$. State stability gives $X^{\circ,N}\to X^\circ$ in $\mathcal S_{\Qq}$. Bounded weights, Cauchy--Schwarz and $|x^\transpose Hx-y^\transpose Hy|\le\|H\|(|x|+|y|)|x-y|$ imply $L^1(\Qq)$ convergence of the total costs in \eqref{eq:J-hat}. Conditional costs converge in $L^1(\Qq)$. For this fixed $u$, select a subsequence such that $\widehat J(t,\xi;u^{N_k})\to\widehat J(t,\xi;u)$ almost surely. If $I$ is the restricted essential infimum, the inequalities $I\le\widehat J(t,\xi;u^{N_k})$ hold together outside a single null set for this sequence. Passing to the limit gives $I\le\widehat J(t,\xi;u)$ almost surely. The null set may depend on $u$; thus $I$ is an essential lower bound for the full family. Domain inclusion gives the reverse inequality, and positive scaling gives \eqref{eq:value-scaling}.
\end{proof}

\section{The transformed value kernel and Riccati identification}\label{sec:transformed}
The transformed problem leads to a stochastic Riccati equation with jumps (SREJ). Starting from the nonrecursive value kernel and dynamic programming principle of \citet{MaMeng2026}, we identify the finite-variation part and the effective control-weight bound.

\begin{assumption}\label{ass:jump-invertible}
There is a deterministic $\varepsilon_E>0$ such that
\begin{equation}\label{eq:jump-invertible}
|\det(\Id_n+E(s,e,\omega))|\ge\varepsilon_E
\quad\dd s\otimes\nu(\dd e)\otimes\dd\Pp\text{-a.e.}
\end{equation}
It is a condition on the uncontrolled state jump map.
\end{assumption}

\begin{assumption}\label{ass:UC}
There is a deterministic constant $\delta>0$ such that, for every $t\in[0,T]$ and $u\in\cU_{\Qq}[t,T]$,
\begin{equation}\label{eq:UC}
\widehat J(t,0;u)\ge\delta\E_t^{\Qq}\int_t^T|u(s)|^2\dd s,\quad\Qq\text{-a.s.}
\end{equation}
\end{assumption}

\begin{remark}\label{rmk:UC-recursive}
By \eqref{eq:J-scaling}, the common-domain form of \assref{ass:UC} is
\[
J^{\mathrm r}(t,0;u)\ge\frac{\delta}{\kappa(t)}
\E_t^{\Qq}\int_t^T|u_s|^2\dd s.
\]
Conversely, the same bound on the common domain extends by truncation and $L^1(\Qq)$ continuity to the full transformed domain. The norm is $L^2(\Qq)$: equivalent probabilities need not give uniformly equivalent global $L^2$ norms.
\end{remark}

For matrix-valued processes $(\widehat P,\widehat\Lambda,\widehat\Xi)$, define the following coefficients. Time arguments are suppressed in longer formulas; $\widehat P_-$ denotes the predictable left-limit kernel.
\begin{equation}\label{eq:Nhat}
\begin{aligned}
\widehat{\mathsf N}:={}&\widehat R+D^\transpose \widehat P_-D\\
&+\int_{\Zmark}F(e)^\transpose[\widehat P_-+\widehat\Xi(e)]F(e)\nu^\circ(\dd e),
\end{aligned}
\end{equation}
\begin{equation}\label{eq:Mhat}
\begin{aligned}
\widehat{\mathsf M}:={}&\widehat P_-B^\circ+\widehat\Lambda D+C^\transpose \widehat P_-D+\widehat S^\transpose\\
&+\int_{\Zmark}\Big[E(e)^\transpose \widehat P_-F(e)\\
&\qquad+(I_n+E(e))^\transpose \widehat\Xi(e)F(e)\Big]\nu^\circ(\dd e),
\end{aligned}
\end{equation}
\begin{equation}\label{eq:Hhat}
\begin{aligned}
\widehat{\mathsf H}:={}&(A^\circ)^\transpose \widehat P_-+\widehat P_-A^\circ+\widehat\Lambda C+C^\transpose \widehat\Lambda\\
&+C^\transpose \widehat P_-C+\widehat Q+\int_{\Zmark}\Big[\widehat\Xi(e)E(e)+E(e)^\transpose \widehat\Xi(e)\\
&\qquad+E(e)^\transpose(\widehat P_-+\widehat\Xi(e))E(e)\Big]\nu^\circ(\dd e).
\end{aligned}
\end{equation}

\begin{definition}\label{def:regular-Riccati}
A symmetric triple $(\widehat P,\widehat\Lambda,\widehat\Xi)$ is a \emph{regular transformed Riccati solution} if $\widehat P$ is adapted and c\`adl\`ag with $\mathop{\rm ess\,sup}_{\omega}\sup_{s\le T}\|\widehat P_s(\omega)\|<\infty$, $\widehat\Lambda,\widehat\Xi$ are predictable with
\begin{equation}\label{eq:Riccati-H2}
\E^{\Qq}\int_0^T\left(|\widehat\Lambda(s)|^2+
\int_{\Zmark}|\widehat\Xi(s,e)|^2\nu^{\circ}(\dd e)\right)\dd s<\infty,
\end{equation}
it solves
\begin{equation}\label{eq:SREJ-hat}
\left\{\begin{aligned}
\dd\widehat P_s={}&-[\widehat{\mathsf H}_s-\widehat{\mathsf M}_s\widehat{\mathsf N}_s^{-1}\widehat{\mathsf M}_s^\transpose]\dd s\\
&+\widehat\Lambda_s\dd W_s^\circ+
\int_{\Zmark}\widehat\Xi(s,e)\widetilde\mu^\circ(\dd s,\dd e),\\
\widehat P_T={}&\widehat G.
\end{aligned}\right.
\end{equation}
and, for a deterministic $c_N>0$,
\begin{equation}\label{eq:Nhat-positive}
\widehat{\mathsf N}(t)\succeq c_N\Id_m
\qquad\dd t\otimes\dd\Qq\text{-a.e.}
\end{equation}
The drift must be pathwise integrable. Feedback admissibility is specified in Section~\ref{sec:feedback}.
\end{definition}

\begin{proposition}\label{thm:nonrecursive}
Let \namedref{Assumptions}{ass:state}, \ref{ass:weights}, \ref{ass:recursive-coeff}, \ref{ass:jump-invertible} and \ref{ass:UC} hold. Then there exists a regular solution $(\widehat P,\widehat\Lambda,\widehat\Xi)$ of \eqref{eq:SREJ-hat}, with $c_N=\delta$ in \eqref{eq:Nhat-positive}. For every deterministic $t<T$ and $\xi\in L^2_{\Qq,\cF_t}(\R^n)$, the transformed conditional problem has a unique open-loop optimizer $\widehat u^{t,\xi,*}$ and
\begin{equation}\label{eq:Vhat-quadratic}
\widehat V(t,\xi)=\ip{\widehat P(t)\xi}{\xi}.
\end{equation}
Its state and control satisfy, for a deterministic $C$ independent of $(t,\xi)$,
\begin{equation}\label{eq:transformed-optimal-estimate}
\E_t^{\Qq}\left[\sup_{t\le s\le T}|\widehat X_s^{t,\xi,*}|^2+
\int_t^T|\widehat u_s^{t,\xi,*}|^2\dd s\right]\le C|\xi|^2.
\end{equation}
The value determines the canonical kernel and its martingale coefficients.
\end{proposition}

\begin{proof}
We use the value representation, dynamic programming and semimartingale construction in \citet[Proposition~3.1 and Theorems~3.1--3.2]{MaMeng2026}. Its coefficient assumption (AS1) follows from \eqref{eq:transformed-drift} and the finite compensator. For its weight assumption (AS2), \eqref{eq:transformed-weights} preserves bounded predictable running weights, a bounded $\cF_T$-measurable terminal weight, and symmetry of $\widehat Q,\widehat R,\widehat G$. Section~\ref{sec:setting} gives the independent Brownian--Poisson filtration. Equivalence preserves \assref{ass:jump-invertible}; replacing $E$ by zero on its predictable exceptional set gives the pointwise lower bound $\min(\varepsilon_E,1)$. Compensation leaves the state integrals unchanged, and the adjugate formula bounds $(I_n+E)^{-1}$.

Taking expectations in \eqref{eq:UC} gives coercivity on the real Hilbert control space. Testing scalar coercivity on $\1_Au$, $A\in\cF_t$, gives the conditional inequality; extending a control on $(\tau,T]$ by zero before $\tau$ and testing on $\cF_\tau$ events gives the stopped bound. Hilbert-space minimization and pasting on $\cF_t$ events yield the unique conditional optimizer. Splitting the state into its zero-control and zero-initial parts bounds the cost's bilinear cross term by $C|\xi|(\E_t^{\Qq}\int_t^T|u|^2\dd s)^{1/2}$. Comparison with zero gives
\[
\delta\E_t^{\Qq}\int_t^T|\widehat u^*|^2\dd s
\le C|\xi|\Big(\E_t^{\Qq}\int_t^T|\widehat u^*|^2\dd s\Big)^{1/2}.
\]
Absorption and the state estimate prove \eqref{eq:transformed-optimal-estimate}.

The cited construction gives a bounded symmetric c\`adl\`ag value kernel. Decompose its predictable finite-variation part as
\begin{equation}\label{eq:value-kernel-before-Riccati}
\begin{aligned}
\dd\widehat P_s={}&\dd\widehat{\mathcal A}_s+\widehat\Lambda_s\dd W_s^\circ
+\int\widehat\Xi(s,e)\widetilde\mu^\circ(\dd s,\dd e),\\
\dd\widehat{\mathcal A}_s={}&\widehat\Psi_s\dd s+\dd\widehat{\mathcal A}^{\perp}_s.
\end{aligned}
\end{equation}
Here $\widehat{\mathcal A}_0=0$, $\widehat{\mathcal A}$ is predictable of finite variation, $\dd|\widehat{\mathcal A}^{\perp}|$ is singular with respect to time, and the martingale integrands are initially local. In particular, $\int_0^T|\widehat\Lambda_s|^2\dd s<\infty$ pathwise. The jumps of $\widehat{\mathcal A}$ form a predictable time-null set, avoided by the Poisson atoms. Hence $\widehat\Xi$ equals $\Delta\widehat P$ at each atom. Boundedness of the kernel, simplicity and compensation give $|\widehat\Xi|\le2\sup|\widehat P|$ in product measure. Thus $\widehat{\mathsf N}$ is bounded and $\int_0^T|\widehat{\mathsf M}|^2\dd s<\infty$ pathwise.

\emph{Step 1: drift nonnegativity.}
For any admissible pair, the dynamic programming principle makes
$X_s^\transpose\widehat P_sX_s+\int_t^s\widehat\ell_r\dd r$
a submartingale, and a martingale along an optimal pair. Joint It\^o calculus, including \eqref{eq:joint-jump-remainder}, gives its predictable finite-variation part as
$X_-^\transpose\dd\widehat{\mathcal A}^{\perp}X_-+\mathcal D_s(X_-,u_s)\dd s$, where
\begin{equation}\label{eq:pre-Riccati-polynomial}
\begin{aligned}
\mathcal D_s(x,v)={}&x^\transpose(\widehat\Psi_s+\widehat{\mathsf H}_s)x
+2x^\transpose\widehat{\mathsf M}_sv+v^\transpose\widehat{\mathsf N}_sv.
\end{aligned}
\end{equation}
These calculations are localized by the state, variation and energy clocks. Uniqueness of the local Doob--Meyer decomposition makes both mutually singular parts nonnegative. For constant $v$, the state is $\Phi_sx+G_s^0v$, where $G^0v$ is its zero-initial part. The uncontrolled flow $\Phi$ is invertible between jumps by the continuous inverse equation and across jumps by $I_n+E$. Applying the drift inequality first to rational $x,v$, then using this affine representation and continuity, gives
\begin{equation}\label{eq:drift-polynomial-positive}
\mathcal D_s(x,v)\ge0\quad\text{for all }(x,v),\qquad
\dd s\otimes\dd\Qq\text{-a.e.}
\end{equation}
The singular parts for $v=0$ similarly give
$\Phi_-^\transpose\dd\widehat{\mathcal A}^{\perp}\Phi_-\succeq0$,
hence $\dd\widehat{\mathcal A}^{\perp}\succeq0$ as a matrix-valued measure. In particular, $\widehat{\mathsf N}\succeq0$.

\emph{Step 2: auxiliary-cost comparison.}
Fix $0<\varepsilon<\delta$ and replace $\widehat R$ by $\widehat R-\varepsilon I_m$. The auxiliary problem is uniformly convex with constant $\delta-\varepsilon$ and has a bounded value kernel $\widehat P^\varepsilon$. Cost comparison at rational vectors and times, followed by right-continuity, gives $\widehat P\succeq\widehat P^\varepsilon$ simultaneously at all times. Their predictable finite-variation parts do not jump at Poisson atoms, so at these atoms
\[
\widehat P_-+\widehat\Xi(e)
-\widehat P^\varepsilon_--\widehat\Xi^\varepsilon(e)\succeq0.
\]
Compensation transfers this order to product measure. Step~1 for the auxiliary problem gives $\widehat{\mathsf N}^\varepsilon\succeq0$, whence
\[
\begin{aligned}
\widehat{\mathsf N}={}&\widehat{\mathsf N}^\varepsilon+\varepsilon I_m
+D^\transpose(\widehat P_--\widehat P^\varepsilon_-)D\\
&+\int F^\transpose\big[\widehat P_-+\widehat\Xi
-\widehat P^\varepsilon_--\widehat\Xi^\varepsilon\big]F\dd\nu^\circ
\succeq\varepsilon I_m.
\end{aligned}
\]
A countable sequence $\varepsilon\uparrow\delta$ gives $\widehat{\mathsf N}\succeq\delta I_m$.

\emph{Step 3: identification between jumps.}
Set $\widehat\Theta=\widehat{\mathsf N}^{-1}\widehat{\mathsf M}^\transpose$ and
$\mathcal K=\widehat\Psi+\widehat{\mathsf H}-\widehat{\mathsf M}\widehat{\mathsf N}^{-1}\widehat{\mathsf M}^\transpose$.
Then $\mathcal K\succeq0$ and $\int_0^T|\widehat\Theta|^2\dd s<\infty$ pathwise. Along an optimal pair the nonnegative finite-variation measure vanishes. Completing $\mathcal D$ and using $\widehat{\mathsf N}\succeq\delta I_m$ therefore gives
\begin{equation}\label{eq:feedback-before-Riccati}
\begin{gathered}
u^*=-\widehat\Theta X_-^*,\quad \mathcal KX_-^*=0
\quad\dd s\otimes\dd\Qq\text{-a.e.},\\
X_-^{*\transpose}\dd\widehat{\mathcal A}^{\perp}X_-^*=0.
\end{gathered}
\end{equation}
For each rational $r\in[0,T)$, including zero, collect the optimal states from $(r,e_i)$ as columns of $Z^r$, with $Z^r_r=I_n$. Let $\sigma_r$ be the first Poisson atom after $r$, capped at $T$. Up to this jump, the closed-loop matrix equation has coefficients
\[
\begin{aligned}
b_s&=A_s^\circ-B_s^\circ\widehat\Theta_s
-\int(E(s,e)-F(s,e)\widehat\Theta_s)\dd\nu^\circ,\\
c_s&=C_s-D_s\widehat\Theta_s.
\end{aligned}
\]
They have pathwise integrable drift and square-integrable diffusion. The continuous inverse equation
$\dd Y=Y(c^2-b)\dd s-Yc\dd W^\circ$, $Y_r=I_n$,
shows that $Z^r_{s-}$ is invertible for $r<s\le\sigma_r$. Thus \eqref{eq:feedback-before-Riccati} implies $\mathcal K=0$ almost everywhere there. Summing the singular identities over the columns gives
$\tr[(Z^r_-)^\transpose\dd\widehat{\mathcal A}^{\perp}Z^r_-]=0$;
positivity and invertibility imply $\dd\widehat{\mathcal A}^{\perp}=0$ on $(r,\sigma_r]$. These countably many intervals cover $(0,T]$ on every finite-activity path, including atom times through left limits and the terminal endpoint. Hence $\widehat{\mathcal A}^{\perp}=0$ and
$\widehat\Psi=-\widehat{\mathsf H}+\widehat{\mathsf M}\widehat{\mathsf N}^{-1}\widehat{\mathsf M}^\transpose$.

To obtain global integrability, put $\widehat{\mathsf L}=\widehat{\mathsf M}\widehat{\mathsf N}^{-1}\widehat{\mathsf M}^\transpose\succeq0$ and $\mathscr Y=\widehat P+(K_0+1)I_n\succeq I_n$, where $K_0$ bounds $\|\widehat P\|$. Apply It\^o's formula to $|\mathscr Y|^2$ at $\rho_j$, the first time $\int_0^s(|\widehat\Lambda|^2+\tr\widehat{\mathsf L})\dd r$ reaches $j$, capped at $T$. Bounded $\mathscr Y,\widehat\Xi$, the bound $|\widehat{\mathsf H}|\le C(1+|\widehat\Lambda|)$ and Young's inequality give
\[
\E^{\Qq}\int_0^{\rho_j}\Big(\tfrac12|\widehat\Lambda|^2
+\int|\widehat\Xi|^2\dd\nu^\circ+2\tr\widehat{\mathsf L}\Big)\dd s\le C.
\]
The stops eventually equal $T$. Fatou gives \eqref{eq:Riccati-H2} and an $L^1$ drift, so the triple is regular. Polarization and right-continuity determine the canonical kernel; uniqueness of its special-semimartingale decomposition determines the integrands.
\end{proof}

For nonrecursive Brownian LQ problems, \citet[Theorem~7.3]{SunXiongYong2021} give a uniform-convexity criterion based on a Lyapunov BSDE. The following elementary bound uses conditional state estimates for the transformed nonrecursive jump problem.

\begin{proposition}\label{prop:coefficient-UC}
Under \namedref{Assumptions}{ass:state}--\ref{ass:recursive-coeff}, let $A_0,\ldots,F_0$ be uniform operator-norm bounds for $A^\circ,B^\circ,C,D,E,F$, respectively. Set $\lambda_\circ=\nu^\circ(\Zmark)$ and
\begin{equation}\label{eq:UC-data-constants}
\begin{aligned}
\alpha_0&=2A_0+1+2C_0^2+2\lambda_\circ E_0^2,\\
\gamma_0&=B_0^2+2D_0^2+2\lambda_\circ F_0^2,\\
C_T&=\gamma_0e^{\alpha_0T},\qquad C_I=TC_T.
\end{aligned}
\end{equation}
Write $H^-$ for the spectral negative part of a symmetric matrix. With $g_-=\|\widehat G^-\|_\infty$, $q_-=\|\widehat Q^-\|_\infty$ and $s_0=\|\widehat S\|_\infty$, if $\widehat R\succeq r_0I_m$ for a deterministic $r_0$ and
\begin{equation}\label{eq:UC-data-test}
\delta_0:=r_0-g_-C_T-q_-C_I-2s_0\sqrt{C_I}>0,
\end{equation}
then \assref{ass:UC} holds with $\delta=\delta_0$.
\end{proposition}
\begin{proof}
For the zero-initial solution of \eqref{eq:state-Q}, the drift of $|X|^2$ under $\Qq$ is bounded by $\alpha_0|X_-|^2+\gamma_0|u|^2$. Localization, the $L^2$ state estimate and conditional Gronwall give, with $U_t=\E_t^{\Qq}\int_t^T|u_s|^2\dd s$,
\[
\E_t^{\Qq}|X_T|^2\le C_TU_t,\qquad
\E_t^{\Qq}\int_t^T|X_s|^2\dd s\le C_IU_t.
\]
The terminal and state terms are bounded below by $-g_-C_TU_t$ and $-q_-C_IU_t$. Conditional Cauchy--Schwarz bounds the cross term in absolute value by $2s_0\sqrt{C_I}U_t$. Adding the control term proves the claim.
\end{proof}
This test allows negative state and terminal weights; \exref{ex:negative-R-jumps} treats a negative control weight.

\begin{lemma}\label{lem:bounded-Xi}
For any regular transformed solution, if $|\widehat P|\le K$ uniformly, then $|\widehat\Xi(t,e)|\le2K$ for $\dd t\otimes\nu^{\circ}(\dd e)\otimes\dd\Qq$-a.e. $(t,e,\omega)$. Consequently $\widehat{\mathsf N}$ is bounded above, and
\begin{equation}\label{eq:Theta-hat}
\begin{aligned}
&\widehat\Theta:=\widehat{\mathsf N}^{-1}\widehat{\mathsf M}^{\transpose}
\quad\hbox{satisfies}\quad
|\widehat\Theta|\le C(1+|\widehat\Lambda|),\\
&\E^{\Qq}\int_0^T|\widehat\Theta(s)|^2\dd s<\infty.
\end{aligned}
\end{equation}
\end{lemma}

\begin{proof}
The Brownian and finite-variation parts of \eqref{eq:SREJ-hat} are continuous, so
$\Delta\widehat P_s=\int\widehat\Xi(s,e)\mu(\{s\},\dd e)$.
By \eqref{eq:time-simple}, each atom value equals this jump and has norm at most $2K$. For the predictable set $A=\{|\widehat\Xi|>2K\}$, compensation gives
\[
\E^{\Qq}\int_0^T\int_{\Zmark}\1_A\nu^\circ(\dd e)\dd s
=\E^{\Qq}\int_0^T\int_{\Zmark}\1_A\mu(\dd s,\dd e)=0.
\]
Thus the bound holds in product measure. Finite intensity and bounded data bound $\widehat{\mathsf N}$ above; all terms in $\widehat{\mathsf M}$ except $\widehat\Lambda D$ are bounded. Its inverse control weight is bounded by \eqref{eq:Nhat-positive}, which proves the gain estimate and square integrability. The argument also applies with operator norms.
\end{proof}

\begin{lemma}\label{lem:riccati-moments}
For any regular transformed solution, set $\widehat{\mathsf L}:=\widehat{\mathsf M}\widehat{\mathsf N}^{-1}\widehat{\mathsf M}^{\transpose}\succeq0$.
There is a deterministic $C$ such that for every stopping time $\tau\le T$,
\begin{align}\label{eq:canonical-BMO}
\E^{\Qq}_\tau\!\int_\tau^T\Big(&|\widehat\Lambda(s)|^2+
\int_{\Zmark}|\widehat\Xi(s,e)|^2\nu^{\circ}(\dd e)
+\tr\widehat{\mathsf L}(s)\Big)\dd s\le C.
\end{align}
The Riccati martingale has bounded mean oscillation (BMO) under $\Qq$. For every finite $p\ge1$,
\begin{align}\label{eq:base-moments}
\E^{\Qq}\Bigg[&\left(\int_0^T|\widehat\Lambda(s)|^2\dd s\right)^p+
\left(\int_0^T\!\!\int_{\Zmark}|\widehat\Xi(s,e)|^2\nu^{\circ}(\dd e)\dd s\right)^p\notag\\
&+\left(\int_0^T|\widehat{\mathsf H}(s)-\widehat{\mathsf L}(s)|\dd s\right)^p\Bigg]<\infty.
\end{align}
The constants depend on the coefficient and kernel bounds, $c_N$, $T$ and $\nu^\circ(\Zmark)$.
\end{lemma}

\begin{proof}
Since $x^\transpose\widehat{\mathsf L}x=|\widehat{\mathsf N}^{-1/2}\widehat{\mathsf M}^{\transpose}x|^2$, we have $\widehat{\mathsf L}\succeq0$. Let $K_0$ bound $\|\widehat P\|$ and put $\mathscr Y=\widehat P+(K_0+1)I_n\succeq I_n$. By \lemref{lem:bounded-Xi}, $|\widehat{\mathsf H}|\le C(1+|\widehat\Lambda|)$. Cyclicity of the trace gives
\[
\tr[(\mathscr Y_--I_n)\widehat{\mathsf L}]
=\tr[\widehat{\mathsf L}^{1/2}(\mathscr Y_--I_n)\widehat{\mathsf L}^{1/2}]\ge0.
\]
For a stopping time $\tau$, let $\rho_j$ be the first time $s\ge\tau$ at which $\int_\tau^s(|\widehat\Lambda|^2+\int|\widehat\Xi|^2\dd\nu^\circ+\tr\widehat{\mathsf L})\dd r$ reaches $j$, capped at $T$. The integrals are pathwise finite since $\widehat{\mathsf M}\in L^2$ and $\widehat{\mathsf N}^{-1}$ is bounded. Thus $\rho_j\uparrow T$, eventually equal to $T$ almost surely. In It\^o's formula for $|\mathscr Y|^2$, the stopped Brownian coefficient is bounded by $C|\widehat\Lambda|$ and the jump coefficient is bounded. Both stopped martingales have zero conditional mean. Therefore
\begin{equation}\label{eq:moment-energy-identity}
\begin{aligned}
&\E_\tau^{\Qq}\int_\tau^{\rho_j}\Big[|\widehat\Lambda|^2+
\int|\widehat\Xi(e)|^2\nu^\circ(\dd e)
+2\tr(\mathscr Y_-\widehat{\mathsf L})\Big]\dd s\\
&=\E_\tau^{\Qq}[|\mathscr Y_{\rho_j}|^2-|\mathscr Y_\tau|^2]
+2\E_\tau^{\Qq}\int_\tau^{\rho_j}\tr(\mathscr Y_-\widehat{\mathsf H})\dd s\\
&\le C+\tfrac12\E_\tau^{\Qq}\int_\tau^{\rho_j}|\widehat\Lambda|^2\dd s.
\end{aligned}
\end{equation}
Absorb the last integral and use conditional Fatou as $j\to\infty$ to obtain \eqref{eq:canonical-BMO}. Conditional isometry and the bounded jumps give the BMO assertion, including the convention that includes the martingale jump at the stopping time.

For the moment step put $A_s=\int_0^s g_r\dd r$, with $g$ the nonnegative integrand in \eqref{eq:canonical-BMO}, and $A_s^{(m)}=A_s\wedge m$. Let $g^{(m)}$ be its time density. Then $\E_\tau^{\Qq}(A_T^{(m)}-A_\tau^{(m)})\le C$, and conditional Tonelli and induction give
\[
\begin{aligned}
&\E_\tau^{\Qq}(A_T^{(m)}-A_\tau^{(m)})^j\\
&=j\E_\tau^{\Qq}\int_\tau^T g_s^{(m)}
\E_s^{\Qq}[(A_T^{(m)}-A_s^{(m)})^{j-1}]\dd s\\
&\le j!C^{j-1}\E_\tau^{\Qq}\int_\tau^T g_s^{(m)}\dd s
\le j!C^j.
\end{aligned}
\]
The tower step follows by testing on $\cF_\tau$ events: their indicators times $\1_{\{\tau\le s\}}g_s^{(m)}$ are $\cF_s$-measurable. Monotone convergence in $m$ and comparison with integer powers give every finite moment. Finally,
\[
\begin{aligned}
\int_0^T|\widehat{\mathsf H}-\widehat{\mathsf L}|\dd s
\le C&+C\sqrt T\Big(\int_0^T|\widehat\Lambda|^2\dd s\Big)^{1/2}\\
&+\int_0^T\tr\widehat{\mathsf L}\dd s,
\end{aligned}
\]
which proves \eqref{eq:base-moments}.
\end{proof}

\section{The recursive stochastic Riccati equation with jumps}\label{sec:rsrej}
We pull the SREJ back to a recursive stochastic Riccati equation with jumps (RSREJ) under $\Pp$.

For the canonical transformed triple in \propref{thm:nonrecursive}, define
\begin{equation}\label{eq:P-pullback}
(P,\Lambda,\Xi)=\kappa^{-1}(\widehat P,\widehat\Lambda,\widehat\Xi).
\end{equation}
For $(P,\Lambda,\Xi)$, use the same time-argument convention and set
\begin{equation}\label{eq:Nrec}
\begin{aligned}
\mathsf N:={}&R+D^\transpose P_-D\\
&+\int_{\Zmark}F(e)^\transpose[P_-+\Xi(e)]F(e)\nu^\circ(\dd e),
\end{aligned}
\end{equation}
\begin{equation}\label{eq:Mrec}
\begin{aligned}
\mathsf M:={}&P_-B^\circ+\Lambda D+C^\transpose P_-D+S^\transpose\\
&+\int_{\Zmark}\Big[E(e)^\transpose P_-F(e)\\
&\qquad+(I_n+E(e))^\transpose \Xi(e)F(e)\Big]\nu^\circ(\dd e),
\end{aligned}
\end{equation}
\begin{equation}\label{eq:Hrec}
\begin{aligned}
\mathsf H:={}&(A^\circ)^\transpose P_-+P_-A^\circ+\Lambda C+C^\transpose \Lambda\\
&+C^\transpose P_-C+Q+\int_{\Zmark}\Big[\Xi(e)E(e)+E(e)^\transpose \Xi(e)\\
&\qquad+E(e)^\transpose(P_-+\Xi(e))E(e)\Big]\nu^\circ(\dd e).
\end{aligned}
\end{equation}

\begin{definition}\label{def:regular-recursive}
A symmetric triple $(P,\Lambda,\Xi)$ is a \emph{regular physical Riccati solution} if $P$ is adapted and c\`adl\`ag, $\Lambda,\Xi$ are predictable, and
\begin{equation}\label{eq:physical-regularity}
\begin{gathered}
\mathop{\rm ess\,sup}_{\omega}\sup_{0\le s\le T}\|P_s(\omega)\|<\infty,\\
\E^{\Pp}\int_0^T\left(|\Lambda_s|^2+
\int_{\Zmark}|\Xi(s,e)|^2\nu(\dd e)\right)\dd s<\infty.
\end{gathered}
\end{equation}
It satisfies, indistinguishably on $[0,T]$, the physical Riccati equation
\begin{equation}\label{eq:RSREJ}
\left\{\begin{aligned}
\dd P_s={}&-\Big[\mathsf H_s+a_sP_{s-}+\beta_s\Lambda_s
-\mathsf M_s\mathsf N_s^{-1}\mathsf M_s^\transpose\\
&\qquad+\int_{\Zmark}\eta(e)\Xi(s,e)\nu(\dd e)\Big]\dd s\\
&+\Lambda_s\dd W_s+
\int_{\Zmark}\Xi(s,e)\widetilde\mu(\dd s,\dd e),\\
P_T={}&G.
\end{aligned}\right.
\end{equation}
and, for a deterministic $c>0$,
\[
\mathsf N_s\succeq cI_m\qquad\dd s\otimes\dd\Pp\text{-a.e.}
\]
The full drift and its integrability condition are
\begin{equation}\label{eq:physical-drift}
\begin{aligned}
\mathfrak F={}&\mathsf H+aP_-+\beta\Lambda
+\int_{\Zmark}\eta(e)\Xi(e)\nu(\dd e)\\
&-\mathsf M\mathsf N^{-1}\mathsf M^\transpose,
\qquad \E^{\Pp}\int_0^T|\mathfrak F_s|\dd s<\infty.
\end{aligned}
\end{equation}
\end{definition}
For comparison and uniqueness we restrict to physical solutions whose $\kappa$-scaled triples also satisfy \defref{def:regular-Riccati}.

\begin{theorem}\label{thm:RSREJ}
Under \namedref{Assumptions}{ass:state}, \ref{ass:weights}, \ref{ass:recursive-coeff}, \ref{ass:jump-invertible} and \ref{ass:UC}, the triple $(P,\Lambda,\Xi)$ in \eqref{eq:P-pullback} is a regular physical Riccati solution of \eqref{eq:RSREJ}. Moreover,
\begin{equation}\label{eq:Nrec-positive}
\mathsf N(t)\succeq \frac{c_N}{\overline\kappa}\Id_m
= c_N e^{-\|a\|_\infty T}\Id_m
\qquad\dd t\otimes\dd\Pp\text{-a.e.},
\end{equation}
and, for every $(t,\xi)\in[0,T]\times\cX_t^2$,
\begin{equation}\label{eq:Vrec-quadratic}
V^{\mathrm r}(t,\xi)=\ip{P(t)\xi}{\xi}.
\end{equation}
The value determines the canonical kernel and its martingale coefficients. \namedref{Propositions}{prop:Riccati-maximality}--\ref{prop:Riccati-uniqueness} give comparison among regular transformed solutions and uniqueness in the feedback-admissible class.
\end{theorem}

\begin{proof}
The scaling relations \eqref{eq:transformed-weights} and \eqref{eq:P-pullback} imply, term by term,
\begin{equation}\label{eq:scale-NMH}
(\widehat{\mathsf N},\widehat{\mathsf M},\widehat{\mathsf H})
=\kappa(\mathsf N,\mathsf M,\mathsf H).
\end{equation}
The continuous scalar $\kappa$ also scales the quadratic Riccati term. Since $\widehat P=\kappa P$ and $\dd\kappa=a\kappa\dd t$,
$\dd\widehat P(t)=\kappa(t)\dd P(t)+a(t)\kappa(t)P(t-)\dd t$.
Rewriting the martingale terms in \eqref{eq:SREJ-hat} under $\Pp$ via $\dd W^{\circ}=\dd W-\beta\dd t$ and $\widetilde\mu^{\circ}=\widetilde\mu-\eta\nu\dd t$, and dividing by $\kappa$ gives \eqref{eq:RSREJ}; the terminal condition follows from $\widehat P(T)=\kappa(T)G$.

From \eqref{eq:Nhat-positive} and \eqref{eq:kappa}, $\mathsf N=\kappa^{-1}\widehat{\mathsf N}\succeq c_N\overline\kappa^{-1}\Id_m$, proving \eqref{eq:Nrec-positive}. For $H=\int_0^T|\widehat\Lambda_s|^2\dd s$ and every finite $p\ge1$, \namedref{Lemmas}{lem:riccati-moments} and \ref{lem:density-moments} give
\[
\E^{\Pp}H^p=\E^{\Qq}[\mathcal E_T^{-1}H^p]
\le(\E^{\Qq}\mathcal E_T^{-2})^{1/2}(\E^{\Qq}H^{2p})^{1/2}<\infty.
\]
Bounded $\kappa^{-1}$ and $\nu\le\varepsilon_\eta^{-1}\nu^\circ$ transfer all finite integrated moments of both pulled-back martingale coefficients to $\Pp$. For the drift use
$\mathfrak F=\kappa^{-1}(\widehat{\mathsf H}-\widehat{\mathsf L})+aP_-+\beta\Lambda+\int\eta\Xi\dd\nu$.
The first time-integral is $L^1(\Pp)$ by its second moment under $\Qq$ and the same H\"older estimate; the others follow from bounded coefficients and the proven energies. Common localization identifies the physical integral equation, verifying \defref{def:regular-recursive}. Finally, \eqref{eq:value-scaling} and \eqref{eq:Vhat-quadratic} give
$V^{\mathrm r}(t,\xi)=\kappa(t)^{-1}\ip{\widehat P(t)\xi}{\xi}=\ip{P(t)\xi}{\xi}$.
\end{proof}

\begin{corollary}\label{cor:benchmarks}
Under the hypotheses of \thmref{thm:RSREJ}:
\begin{enumerate}
\item If $a=\beta=\eta=0$, equation \eqref{eq:RSREJ} reduces exactly to the nonrecursive SREJ with the original coefficients and compensator $\nu$.
\item If $\beta=\eta=0$ but $a\ne0$, then $A^{\circ}=A$, $B^{\circ}=B$, $\nu^{\circ}=\nu$, and the only new Riccati drift term is $aP$.
\item The recursive and transformed feedback gains coincide:
\begin{equation}\label{eq:gain-equality}
\Theta(t):=\mathsf N(t)^{-1}\mathsf M(t)^\transpose
=\widehat{\mathsf N}(t)^{-1}\widehat{\mathsf M}(t)^\transpose=\widehat\Theta(t).
\end{equation}
\end{enumerate}
\end{corollary}

\begin{proof}
Parts (1)--(2) follow by substitution in \eqref{eq:RSREJ}. For (3), scalar scaling in \eqref{eq:scale-NMH} gives $(\kappa\mathsf N)^{-1}(\kappa\mathsf M)^\transpose=\mathsf N^{-1}\mathsf M^\transpose$.
\end{proof}

\begin{remark}\label{rmk:LiYong-reduction}
With no Poisson noise and deterministic data, a deterministic Riccati solution has $\Lambda=0$. Writing
$\mathsf N_0=R+D^\transpose PD$ and
$\mathsf M_0=P(B+\beta D)+C^\transpose PD+S^\transpose$,
equation \eqref{eq:RSREJ} takes the homogeneous, nonsingular form of \citet{LiYong2026}, with backward coefficients $a,\beta$:
\begin{equation}\label{eq:deterministic-recursive-Riccati}
\left\{\begin{aligned}
0={}&\dot P+PA+A^\transpose P+C^\transpose PC+Q+aP\\
&+\beta(PC+C^\transpose P)-\mathsf M_0\mathsf N_0^{-1}\mathsf M_0^\transpose,\\
P(T)={}&G.
\end{aligned}\right.
\end{equation}
\end{remark}

\section{Verification and feedback synthesis}\label{sec:feedback}

We first establish completion of squares and identify the canonical feedback. We then characterize physical admissibility and derive comparison and uniqueness.

\subsection{Completion of squares and feedback representation}

\begin{theorem}\label{thm:completion-squares}
Under \namedref{Assumptions}{ass:state}, \ref{ass:weights} and \ref{ass:recursive-coeff}, let a regular transformed Riccati solution be given. For every deterministic $t<T$, $\xi\in L^2_{\Qq,\cF_t}(\R^n)$ and $u\in\cU_{\Qq}[t,T]$, let $X$ solve \eqref{eq:state-Q} and put $w=u+\widehat\Theta X_-$. Then
\begin{equation}\label{eq:completion-Q}
\widehat J(t,\xi;u)=\ip{\widehat P_t\xi}{\xi}
+\E_t^{\Qq}\int_t^T w_s^\transpose\widehat{\mathsf N}_sw_s\dd s.
\end{equation}
The nonnegative integral belongs to $L^1(\Qq)$. In particular, under \thmref{thm:RSREJ}, for $\xi\in\cX_t^2$ and $u\in\cU_{\Pp,\Qq}$,
\begin{equation}\label{eq:completion-squares}
\begin{aligned}
J^{\rm r}(t,\xi;u)=\ip{P_t\xi}{\xi}
+\kappa(t)^{-1}\E_t^{\Qq}\int_t^T
\kappa_s\Delta u_s^\transpose\mathsf N_s\Delta u_s\dd s,
\end{aligned}
\end{equation}
where
$\Delta u=u+\Theta X_-$ and $\Theta=\mathsf N^{-1}\mathsf M^\transpose=\widehat\Theta$.
\end{theorem}
\begin{proof}
Write $x=X_-$, $\Sigma=Cx+Du$ and $\Gamma(e)=E(e)x+F(e)u$. For $f(P,x)=x^\transpose Px$, the raw joint jump is
\[
L(e)=(x+\Gamma(e))^\transpose(\widehat P_-+\widehat\Xi(e))
(x+\Gamma(e))-x^\transpose\widehat P_-x.
\]
Suppressing the mark argument, the jump contribution to the drift in It\^o's formula is
\begin{equation}\label{eq:joint-jump-remainder}
L-2x^\transpose\widehat P_-\Gamma-x^\transpose\widehat\Xi x
=\Gamma^\transpose(\widehat P_-+\widehat\Xi)\Gamma+2x^\transpose\widehat\Xi\Gamma.
\end{equation}
Together with the continuous covariation, this gives the matrices \eqref{eq:Nhat}--\eqref{eq:Hhat}. Substitution of \eqref{eq:SREJ-hat} completes the square:
\begin{equation}\label{eq:ito-square-differential}
\dd(X_s^\transpose\widehat P_sX_s)+\widehat\ell_s\dd s
=w_s^\transpose\widehat{\mathsf N}_sw_s\dd s+\dd\mathcal M_s.
\end{equation}
Here $\widehat\ell_s=\widehat\ell(s,X_{s-},u_s)$ and $\mathcal M_t=0$. The Brownian coefficient of $\mathcal M$ is $K=x^\transpose\widehat\Lambda x+2x^\transpose\widehat P_-\Sigma$, and its jump coefficient is $L$ above. Pathwise boundedness of the c\`adl\`ag state gives $\int|K|^2\dd s<\infty$ almost surely. By boundedness of $\widehat P,\widehat\Xi,E,F$,
\[
|L|\le C(|x|^2+|u|^2),\quad
\E^{\Qq}\int_t^T\int|L|\dd\nu^\circ\dd s<\infty.
\]
Thus the jump integral is a uniformly integrable martingale.

Set $\mathcal B_s=X_s^\transpose\widehat P_sX_s+\int_t^s\widehat\ell_r\dd r$ and $A_s=\int_t^s w_r^\transpose\widehat{\mathsf N}_rw_r\dd r$. Initially $A_T$ is pathwise finite, since $\widehat{\mathsf N}$ is bounded and $\int|\widehat\Theta|^2<\infty$ almost surely. Also, simultaneously for all $s$,
\[
|\mathcal B_s|\le\mathcal D:=C\sup_{t\le r\le T}|X_r|^2+C\int_t^T|u_r|^2\dd r\in L^1(\Qq).
\]
Stop the continuous clock $A_s+\int_t^s|K_r|^2\dd r$ at $j$, capped at $T$. Then the stopped Brownian integral has zero mean, as does the jump integral. The stopped identity gives $0\le\E^{\Qq}A_{\rho_j}\le2\E^{\Qq}\mathcal D$. Monotone convergence yields $A_T\in L^1(\Qq)$. Now for every stopping time $\sigma\in[t,T]$,
\begin{equation}\label{eq:completion-UI-domination}
|\mathcal M_\sigma|=|\mathcal B_\sigma-\mathcal B_t-A_\sigma|
\le2\mathcal D+A_T\in L^1(\Qq).
\end{equation}
The same dominating variable works for all $\sigma$, so the local martingale is of class $(D)$ and uniformly integrable. Conditional expectation at $T$ proves \eqref{eq:completion-Q}. Scalar scaling gives \eqref{eq:completion-squares} on the common domain.
\end{proof}

\begin{corollary}\label{cor:canonical-feedback}
Under \propref{thm:nonrecursive}, the existing transformed optimizer satisfies
\begin{equation}\label{eq:canonical-feedback-Q}
\widehat u_s^{t,\xi,*}=-\widehat\Theta_s\widehat X_{s-}^{t,\xi,*}
\quad\dd s\otimes\dd\Qq\text{-a.e.}
\end{equation}
\end{corollary}
\begin{proof}
Equation \eqref{eq:feedback-before-Riccati} gives the identity. Compensation of its predictable exceptional set makes it valid in the jump state integral as well.
\end{proof}

\subsection{Physical admissibility and attainment}

\begin{definition}\label{ass:cross-admissible}
For a fixed $(t,\xi)$ with $\xi\in\cX_t^2$, the canonical transformed optimizer is physically admissible if
\begin{equation}\label{eq:canonical-P-integrability}
\E^{\Pp}\int_t^T|\widehat u^{t,\xi,*}(s)|^2\dd s<\infty.
\end{equation}
\end{definition}
By \lemref{lem:space-transfer}, $\widehat u^{t,\xi,*}\in\cU^{2p}_{\Qq}[t,T]$ for some $p>1$ is sufficient.

\begin{theorem}\label{thm:feedback}
Under the assumptions of \thmref{thm:RSREJ}, fix a deterministic $t<T$ and $\xi\in\cX_t^2$. The recursive problem on $\cU_{\Pp,\Qq}[t,T]$ has an optimal control if and only if \eqref{eq:canonical-P-integrability} holds. In that case $u^*=\widehat u^{t,\xi,*}$ is the unique optimizer in $\cU_{\Pp,\Qq}[t,T]$, its state belongs to $\mathcal S_{\Pp}[t,T]\cap\mathcal S_{\Qq}[t,T]$, and
\begin{equation}\label{eq:feedback-value}
J^{\mathrm r}(t,\xi;u^*)=V^{\mathrm r}(t,\xi)=\ip{P(t)\xi}{\xi}.
\end{equation}
Its feedback representation is
\begin{equation}\label{eq:recursive-feedback-law}
u^*(s)=-\Theta(s)X^*(s-),
\end{equation}
and $X^*$ is the pathwise unique solution of
\begin{equation}\label{eq:closed-loop-P}
\left\{\begin{aligned}
\dd X_s^*={}&(A_s-B_s\Theta_s)X_{s-}^*\dd s\\
&+(C_s-D_s\Theta_s)X_{s-}^*\dd W_s\\
&+\int_{\Zmark}[E(s,e)-F(s,e)\Theta_s]X_{s-}^*
\widetilde\mu(\dd s,\dd e),\\
X_t^*={}&\xi.
\end{aligned}\right.
\end{equation}

\end{theorem}

\begin{proof}
By \propref{thm:nonrecursive}, the transformed optimizer is unique. Equality of infima and positive cost scaling identify any common-domain minimizer with it, proving necessity of \eqref{eq:canonical-P-integrability}. Conversely, this condition places the optimizer in both control spaces. Solving the state equation under each measure, the noise substitution and pathwise uniqueness identify the states; their $L^2$ estimates give the claimed state spaces. Cost scaling proves \eqref{eq:feedback-value} and transfers uniqueness.

\corref{cor:canonical-feedback} gives \eqref{eq:canonical-feedback-Q}. Gain invariance then yields \eqref{eq:recursive-feedback-law} and \eqref{eq:closed-loop-P}. Pathwise uniqueness follows by localization: \lemref{lem:bounded-Xi} gives $\int_t^T|\Theta|^2\dd s<\infty$ almost surely, and equivalence gives the same statement under $\Pp$. Between Poisson jumps its drift is $A-B\Theta-\int(E-F\Theta)\dd\nu$ and its diffusion is $C-D\Theta$, with the required time integrability; at a jump its update is a specified linear map. Finite activity completes the pathwise argument. The open-loop estimates above give moment integrability.
\end{proof}

Bounded gain verifies \eqref{eq:canonical-P-integrability}; \corref{cor:Dzero} gives a coefficient condition ensuring this bound.

\begin{proposition}\label{prop:bounded-gain}
Under the assumptions of \thmref{thm:RSREJ}, fix a deterministic $t<T$ and suppose $\Theta$ is essentially bounded on $[t,T]\times\Omega$ and $\xi\in\cX_t^2$. Then \eqref{eq:canonical-P-integrability} holds, and the unique recursive optimizer is $u^*=-\Theta X^*_-$ with $X^*$ solving \eqref{eq:closed-loop-P}.
\end{proposition}

\begin{proof}
Bounded closed-loop coefficients give an $\mathcal S^2$ solution under each probability; noise substitution and pathwise uniqueness identify them. Bounded $\Theta$ makes $u^*=-\Theta X^*_-$ a common-domain control. Its square in \thmref{thm:completion-squares} vanishes, proving optimality. Equality of values and transformed uniqueness identify $u^*$ with $\widehat u^{t,\xi,*}$ and give \eqref{eq:canonical-P-integrability}.
\end{proof}

\begin{corollary}\label{cor:Dzero}
Under the assumptions of \thmref{thm:RSREJ}, if $D=0$ almost everywhere, then $\Theta$ is essentially bounded. For every deterministic $t<T$ and $\xi\in\cX_t^2$, the recursive problem has a unique optimal control, given by $u^*=-\Theta X^*_-$. State-dependent Brownian noise and controlled jumps are still allowed.
\end{corollary}
\begin{proof}
The only potentially unbounded term in \eqref{eq:Mhat} is $\widehat\Lambda D$. With $D=0$, boundedness of the kernel and \lemref{lem:bounded-Xi} bound every term. More explicitly, if $\|\widehat P\|\le K$, $E_*=\|E\|_\infty$ and $F_*=\|F\|_\infty$, then
\[
\|\Theta\|_\infty\le\delta^{-1}\big[K\|B^\circ\|_\infty+
\overline\kappa\|S\|_\infty+K(2+3E_*)F_*\nu^\circ(\Zmark)\big].
\]
The conclusion follows from \propref{prop:bounded-gain}.
\end{proof}

\begin{remark}\label{rmk:verification-standalone}
For a regular transformed solution, a $\Qq$-admissible feedback makes the square in \eqref{eq:completion-Q} vanish and attains the transformed value. If the feedback is also $\Pp$-admissible, scaling gives the recursive optimum. This form of \thmref{thm:completion-squares} directly verifies the explicit solutions below.
\end{remark}

\subsection{Comparison and uniqueness}

\begin{proposition}\label{prop:Riccati-maximality}
Under \propref{thm:nonrecursive}, every other regular transformed solution $(\widehat P',\widehat\Lambda',\widehat\Xi')$ for the same data satisfies $\widehat P'(s)\preceq\widehat P(s)$ for all $s\in[0,T]$ almost surely. The corresponding pulled-back kernels have the same order.
\end{proposition}
\begin{proof}
For any regular candidate, \eqref{eq:completion-Q} gives $x^\transpose\widehat P'(t)x\le\widehat J(t,x;u)$ for every transformed admissible control. Taking the conditional infimum gives the order. A countable dense set of vectors and times, right-continuity, and the common terminal value give the order simultaneously on $[0,T]$. Division by positive $\kappa$ preserves it.
\end{proof}

\begin{proposition}\label{prop:Riccati-uniqueness}
Under \namedref{Assumptions}{ass:state}, \ref{ass:weights} and \ref{ass:recursive-coeff}, there is at most one regular transformed solution whose closed-loop state/control pair is square integrable under $\Qq$ for every deterministic initial time and vector. The pulled-back equation likewise has at most one solution whose scaling belongs to this class. A bounded gain is sufficient for this admissibility.
\end{proposition}

\begin{proof}
An admissible candidate feedback makes the square in \eqref{eq:completion-Q} zero, identifying its kernel with the value. Polarization at rational vectors and times, right-continuity, and the common terminal value give indistinguishability. Uniqueness of the semimartingale decomposition identifies both integrands, and scaling transfers the result.
\end{proof}

\section{Analytical examples}\label{sec:finance}

We first consider an unconstrained scalar market with bounded predictable $r,\sigma,c$ and $1+c(s,e)\ge\varepsilon_c>0$. The bank account satisfies $\dd S_0=rS_0\dd s$. A risky asset with $S_1(0)>0$ and zero physical excess return satisfies
\begin{equation}\label{eq:risky-asset}
\frac{\dd S_1(s)}{S_1(s-)}=r(s)\dd s+\sigma(s)\dd W(s)+\int_{\Zmark}c(s,e)\widetilde\mu(\dd s,\dd e).
\end{equation}
Its stochastic exponential has positive jump multipliers $1+c$, hence positive price. With dollar exposure $u$ in the risky asset, wealth satisfies
\begin{equation}\label{eq:wealth-P}
\dd X_s=r_sX_s\dd s+\sigma_su_s\dd W_s+
\int_{\Zmark}c(s,e)u_s\widetilde\mu(\dd s,\dd e).
\end{equation}
Thus $A=r$, $B=C=E=0$, $D=\sigma$ and $F=c$.
Consider the recursive criterion generated by \eqref{eq:recursive-BSDE} with
$G=-\lambda/2$, $Q=S=0$, $R=\rho/2$ and $\lambda>0$, where $\rho$ is a bounded positive running-penalty process. The negative terminal weight rewards large absolute terminal wealth; the example therefore represents an indefinite quadratic criterion rather than risk-averse utility. Since $E=0$, \eqref{eq:jump-invertible} holds automatically.

Under $\Qq$, the wealth drift becomes
\begin{equation}\label{eq:wealth-Q}
\begin{aligned}
\dd X_s={}&(r_sX_s+\vartheta_su_s)\dd s+\sigma_su_s\dd W_s^\circ\\
&+\int_{\Zmark}c(s,e)u_s\widetilde\mu^\circ(\dd s,\dd e).
\end{aligned}
\end{equation}
where
\begin{equation}\label{eq:vartheta}
\vartheta(s):=\beta(s)\sigma(s)+\int_{\Zmark}\eta(e)c(s,e)\nu(\dd e).
\end{equation}
The aggregators produce an effective valuation-measure drift despite zero physical excess return. Its correction uses $\nu$, whereas the quadratic-variation terms below use $\nu^\circ$.

For this market, the physical Riccati coefficients are
\begin{equation}\label{eq:portfolio-Riccati-coefficients}
\begin{aligned}
\mathsf N={}&\frac{\rho}{2}+\sigma^2P_-+
\int_{\Zmark}c(e)^2[P_-+\Xi(e)]\nu^\circ(\dd e),\\
\mathsf M={}&\vartheta P_-+\sigma\Lambda+
\int_{\Zmark}c(e)\Xi(e)\nu^\circ(\dd e),\\
\mathsf H={}&2rP_-.
\end{aligned}
\end{equation}
The scalar physical Riccati equation is
\begin{equation}\label{eq:portfolio-RSREJ}
\left\{\begin{aligned}
\dd P={}&-\Big[(2r+a)P_-+\beta\Lambda+
\int_{\Zmark}\eta(e)\Xi(e)\nu(\dd e)\\
&\qquad-\frac{\mathsf M^2}{\mathsf N}\Big]\dd s
+\Lambda\dd W+\int_{\Zmark}\Xi(e)\widetilde\mu(\dd s,\dd e),\\
P(T)={}&-\lambda/2.
\end{aligned}\right.
\end{equation}
For a physically admissible canonical optimizer, the portfolio feedback is
\[
u^*=-\frac{\mathsf M}{\mathsf N}X^*_-.
\]
The following cases also admit direct verification.

\subsection{An exact drift-neutral convexity benchmark}

\begin{proposition}\label{prop:portfolio-UC}
Assume that $r,a$ are deterministic and bounded and that the bounded predictable market coefficients satisfy $\vartheta=0$ in \eqref{eq:vartheta}. The processes $\sigma,\rho,c$ may still be random. Put
\begin{align}\label{eq:K-v}
K(s,T)&:=\exp\left\{\int_s^T r(q)\dd q\right\},\notag\\
v^{\circ}(s)&:=\sigma(s)^2+\int_{\Zmark}c(s,e)^2\nu^{\circ}(\dd e),\\
h(s)&:=\kappa(s)\rho(s)-\lambda\kappa(T)K(s,T)^2v^{\circ}(s).
\label{eq:h-benchmark}
\end{align}
For any $t<T$, zero initial state and $u\in\cU_{\Pp,\Qq}[t,T]$,
\begin{equation}\label{eq:portfolio-cost-explicit}
J^{\mathrm r}(t,0;u)=\frac1{2\kappa(t)}
\E^{\Qq}\left[\int_t^Th(s)|u(s)|^2\dd s\,\Big|\,\cF_t\right].
\end{equation}
For every $u\in\cU_{\Qq}[t,T]$, the transformed identity $\widehat J(t,0;u)=\frac12\E_t^{\Qq}\int_t^T h(s)|u_s|^2\dd s$ holds. Uniform convexity of that full problem on $[0,T]$ is equivalent to
\begin{equation}\label{eq:portfolio-UC-exact}
\operatorname*{ess\,inf}_{(s,\omega)\in[0,T]\times\Omega}h(s,\omega)>0,
\quad\hbox{relative to }\dd s\otimes\dd\Qq.
\end{equation}
When $a,r,\rho,v^{\circ}$ are constants, this is equivalent to
\begin{equation}\label{eq:portfolio-constant-condition}
\rho>\lambda v^{\circ}\exp\{\max(a+2r,0)T\}.
\end{equation}
Under \eqref{eq:portfolio-UC-exact}, for $\xi\in\cX_t^2$, the value and unique optimizer on the common domain are
\begin{equation}\label{eq:neutral-value}
V^{\rm r}(t,\xi)=-\frac\lambda2e^{\int_t^T[a(s)+2r(s)]\dd s}\xi^2,\qquad u^*=0.
\end{equation}
\end{proposition}

\begin{proof}
For zero initial state and $u\in\cU_{\Qq}[t,T]$, the terminal state is the sum of the Brownian and compensated-Poisson integrals with integrands $K(s,T)\sigma_su_s$ and $K(s,T)c(s,e)u_s$. Their orthogonality and conditional isometry give
\[
\E_t^{\Qq}|X_T|^2
=\E_t^{\Qq}\int_t^TK(s,T)^2v^\circ(s)|u_s|^2\dd s.
\]
Deterministic $\kappa(T)$ gives the transformed cost identity, and \eqref{eq:J-scaling} gives \eqref{eq:portfolio-cost-explicit}. A bound $h\ge\varepsilon>0$ implies uniform convexity. Conversely, testing $u=\1_{\{h<2\delta\}}$ at $t=0$ shows that uniform convexity with constant $\delta$ requires $h\ge2\delta$ in product measure. For constant data, divide $h(s)$ by $\kappa(s)$ and maximize $e^{(a+2r)(T-s)}$ to obtain \eqref{eq:portfolio-constant-condition}.

For general $\xi$, $X_T=K(t,T)\xi+M_T$, where $\E_t^{\Qq}M_T=0$ and $M$ is the zero-initial martingale integral. Thus
\[
\begin{aligned}
J^{\mathrm r}(t,\xi;u)={}&-\frac{\lambda\kappa(T)K(t,T)^2}{2\kappa(t)}\xi^2\\
&+\frac{1}{2\kappa(t)}\E_t^{\Qq}\int_t^Th(s)|u_s|^2\dd s.
\end{aligned}
\]
This proves \eqref{eq:neutral-value} and uniqueness. Its Riccati triple is $P(s)=-\frac\lambda2e^{\int_s^T(a+2r)\dd q}$, $\Lambda=\Xi=0$, with $\mathsf M=0$, $\mathsf N=\rho/2+Pv^\circ$ and $2\kappa\mathsf N=h$.
\end{proof}

\begin{remark}\label{rmk:finance-boundaries}
If $h<0$ on a set of positive product measure, scaling a bounded predictable control supported there drives the unconditional zero-initial transformed cost to $-\infty$; if $h\ge0$ but its essential infimum is zero, the zero-initial cost is nonnegative but not uniformly convex. Keeping $r$ deterministic, when $\vartheta\ne0$ the terminal wealth also contains $\int K(s,T)\vartheta(s)u(s)\dd s$, whose square is not a pointwise multiple of $|u(s)|^2$; one must then check the transformed quadratic functional or establish a Riccati verification with admissibility.
\end{remark}

\subsection{A recursive example with nonzero feedback}

\begin{example}\label{ex:nonzero-feedback}
Take a unit-rate, single-mark Poisson process and constants $\sigma>0$, $c>-1$, $\beta,\eta$ with $\eta>-1$. Let $r=a=0$, $\vartheta=\beta\sigma+\eta c\ne0$ and $v^\circ=\sigma^2+(1+\eta)c^2$. For $k,\lambda>0$, choose
$\rho(s)=\lambda(v^{\circ}+\vartheta^2/k)e^{k(T-s)}$.
With $\kappa=1$, equation \eqref{eq:RSREJ} has the solution
\begin{align}\label{eq:constructed-solution}
P(s)&=-\frac\lambda2e^{k(T-s)},\qquad\Lambda=\Xi=0,\notag\\
\mathsf N(s)&=\frac{\lambda\vartheta^2}{2k}e^{k(T-s)}>0,
\qquad\mathsf M(s)=\vartheta P(s),\notag\\
\Theta(s)&=-\frac{k}{\vartheta},
\qquad u^*(s)=\frac{k}{\vartheta}X^*(s-).
\end{align}
Indeed, $\dot P=-kP=\mathsf M^2/\mathsf N$ and $P(T)=-\lambda/2$. Here $D=\sigma>0$. Bounded gain gives admissibility under both measures for $\xi\in\cX_t^2$; \remref{rmk:verification-standalone} then yields $V^{\mathrm r}(t,\xi)=P(t)\xi^2$. For zero initial state, set $w=u+\Theta X_-$. The bounded coefficients of the state equation driven by $w$ imply
\[
\E_t^{\Qq}\int_t^T|u|^2\dd s
\le C\E_t^{\Qq}\int_t^T|w|^2\dd s.
\]
Together with $\inf_s\mathsf N(s)>0$ and \eqref{eq:completion-Q}, this proves uniform convexity.
\end{example}

\subsection{A pure-jump problem with negative control weight}

\begin{example}\label{ex:negative-R-jumps}
Let $N$ be a unit-rate Poisson process under $\Pp$, with one mark. Take $A=B=C=D=a=\beta=0$, $E=F=\eta=1$, $Q=2$, $R=-1/2$, $S=0$ and $G=1$.
The equations are
\begin{equation}\label{eq:negative-R-model}
\begin{aligned}
\dd X_s&=(X_{s-}+u_s)(\dd N_s-\dd s),\quad X_t=\xi,\\
-\dd Y_s&=(2X_{s-}^2-\tfrac12u_s^2+U_s)\dd s\\
&\quad-Z_s\dd W_s-U_s(\dd N_s-\dd s),\quad Y_T=X_T^2.
\end{aligned}
\end{equation}
Here $\mathcal E_s=2^{N_s}e^{-s}$, $\kappa=1$, and $N$ has intensity $2$ under $\Qq$. Thus $A^\circ=B^\circ=1$. The constant Riccati solution is
\begin{equation}\label{eq:negative-R-kernel}
\begin{gathered}
P=\widehat P=1,\quad\Lambda=\Xi=0,\\
\mathsf N=\tfrac32,\quad\mathsf M=3,\quad\mathsf H=6,\quad\Theta=2.
\end{gathered}
\end{equation}
Indeed, $\mathsf H-\mathsf M^2/\mathsf N=0$ and $P_T=G$. For $u\in\cU_{\Qq}[t,T]$, It\^o's formula for $X^2$ gives
\begin{equation}\label{eq:negative-R-square}
\widehat J(t,\xi;u)=\xi^2+\tfrac32\E_t^{\Qq}
\int_t^T(u_s+2X_{s-})^2\dd s.
\end{equation}
At zero initial state, put $w=u+2X_-$. The drift of $X^2$ is then $2w^2-2X_-w\le X_-^2+3w^2$. Conditional Gronwall and $u=w-2X_-$ give
\[
\E_t^{\Qq}\int_t^T|u_s|^2\dd s
\le(2+24Te^T)\E_t^{\Qq}\int_t^T|w_s|^2\dd s.
\]
Hence \assref{ass:UC} holds with $\delta=3/(4+48Te^T)>0$, although $R<0$.

The optimal pair and value are
\begin{equation}\label{eq:negative-R-optimum}
\begin{aligned}
u_s^*&=-2X_{s-}^*,\qquad
X_s^*=\xi e^{s-t}\1_{\{N_s-N_t=0\}},\\
V^{\mathrm r}(t,\xi)&=\xi^2.
\end{aligned}
\end{equation}
For $\xi\in\cX_t^2$, the bound $\sup_s|X_s^*|\le e^T|\xi|$ gives admissibility under both measures, and \eqref{eq:negative-R-square} proves optimality and uniqueness. Here $D=0$ and $1+E=2$, while the closed-loop multiplier $1+E-F\Theta=0$ resets the state at a jump. Adding a second control with cost $u_2^2/2$ and no state effect gives $R=\operatorname{diag}(-1/2,1/2)$. Removing the jump channel while retaining the original weights makes the zero-initial cost unbounded below, as seen by scaling deterministic controls. The jump channel is therefore essential to uniform convexity in this example.
\end{example}

\section{Conclusion}\label{sec:conclusion}
The common control domain gives $L^1$ recursive evaluation and exact value reduction. Cost comparison and restarted optimal trajectories identify the Riccati equation; conditional energy estimates justify its physical-measure pullback. Completion of squares yields feedback, comparison, uniqueness in the feedback-admissible class and an attainment criterion. Bounded gain ensures attainment. Further work may address weaker admissibility conditions, random aggregators, infinite activity, affine data and singular control weights.

\appendix
\section{Physical-measure first-order verification}\label{app:direct-check}

We express the feedback in physical-measure adjoint variables and verify its first variation. For related maximum-principle methods for mean-field jump diffusions with delay, see \citet{ShenMengShi2014}.

For a fixed initial time $t$, define the positive recursive multiplier
$\chi^t(s):=\kappa(s)\mathcal E(s)/(\kappa(t)\mathcal E(t))$, $s\in[t,T]$.
Its drift is $a\chi^t$, its Brownian coefficient is $\beta\chi^t_-$, and its jump coefficient is $\eta\chi^t_-$. For the optimal pair and its bounded predictable variations, the moment bounds below justify Bayes' formula and conditional Fubini, giving
\begin{equation}\label{eq:physical-cost-representation}
\begin{aligned}
J^{\rm r}(t,\xi;u)=\E_t^{\Pp}\Big[&\chi^t(T)\ip{GX_T}{X_T}\\
&+\int_t^T\chi^t(s)\ell(s,X_{s-},u_s)\dd s\Big].
\end{aligned}
\end{equation}
Here the initial state and feedback gain are bounded, as specified in the next proposition.

\begin{proposition}\label{prop:direct-adjoint}
Let the assumptions of \thmref{thm:RSREJ} hold.  Assume in addition that $\Theta$ is essentially bounded and $\xi\in L^\infty_{\cF_t}(\R^n)$.  Then \propref{prop:bounded-gain} supplies the optimal pair $(X^*,u^*)$, satisfying \eqref{eq:canonical-P-integrability}. Put $\Sigma^*=CX^*_-+Du^*$ and $\Gamma^*(e)=E(e)X^*_-+F(e)u^*$. Suppressing time arguments, define
\begin{equation}\label{eq:adjoint-definitions}
\begin{aligned}
p={}&2\chi^t PX^*,\\
q={}&2\chi^t_-[\Lambda X^*_-+P_-\Sigma^*+\beta P_-X^*_-],\\
\varrho(e)={}&2\chi^t_-\big[(1+\eta(e))(P_-+\Xi(e))\\
&\qquad\times(X^*_-+\Gamma^*(e))-P_-X^*_-\big].
\end{aligned}
\end{equation}
They satisfy
\begin{equation}\label{eq:physical-adjoint}
\left\{
\begin{aligned}
\dd p={}&-\Big[A^\transpose p_-+C^\transpose q
+\int_{\Zmark}E(e)^\transpose\varrho(e)\nu(\dd e)\\
&\qquad+2\chi^t\big(QX^*_-+S^\transpose u^*\big)\Big]\dd s\\
&+q\dd W+\int_{\Zmark}\varrho(e)\widetilde\mu(\dd s,\dd e),\\
p(T)={}&2\chi^t(T)GX^*(T).
\end{aligned}
\right.
\end{equation}
and the stationarity condition
\begin{equation}\label{eq:physical-stationarity}
\begin{aligned}
&B^\transpose p_-+D^\transpose q+
\int_{\Zmark}F(e)^\transpose\varrho(e)\nu(\dd e)\\
&\qquad+2\chi^t\big(SX^*_-+Ru^*\big)=0
\end{aligned}
\end{equation}
holds $\dd s\otimes\dd\Pp$-a.e.  Moreover, for every bounded predictable direction $v$, the G\^ateaux derivative of the original recursive criterion satisfies
$DJ^{\mathrm r}(t,\xi;u^*)[v]=0$ in $L^1(\Omega,\cF_t,\Pp)$.
\end{proposition}

\begin{proof}
Bounded gain, initial state and direction $v$ give all finite moments of $\sup|X^*|$, $\sup|X_1|$ and $\mathop{\rm ess\,sup}|u^*|$, where $X_1=X^{t,0;v}$. Density-ratio suprema have the same property by \lemref{lem:density-moments} and maximal inequalities. The Riccati moments from \thmref{thm:RSREJ}, H\"older's inequality and \lemref{lem:bounded-Xi} give $p\in\mathcal S^2_{\Pp}$, globally square-integrable $q,\varrho$, and integrability of the products and variations below.

Replace $(X^*,u^*)$ by $(X,u)$ in \eqref{eq:adjoint-definitions} to define $(p^u,q^u,\varrho^u)$. Localizing the state, multiplier and energy clocks, apply the product rule to $2\chi^t PX$. With $b_X=AX_-+Bu$, $b_P=-\mathfrak F$, $\Sigma=CX_-+Du$ and $\Gamma=EX_-+Fu$, the drift divided by $2\chi^t_-$ is
\[
\begin{aligned}
&aPX_-+b_PX_-+Pb_X+\Lambda\Sigma
+\beta(\Lambda X_-+P\Sigma)\\
&\quad+\int_{\Zmark}\Big[\Xi\Gamma+
\eta(P\Gamma+\Xi X_-+\Xi\Gamma)\Big]\dd\nu.
\end{aligned}
\]
Here $P=P_-$. The continuous covariations include $\beta P\Sigma$, and the integral includes the triple jump product. Subtracting the adjoint drift and dividing by $2\chi^t_-$ leaves the coefficient $\mathsf M$ on $u$. On $X_-$, the Riccati equation gives
$\mathsf H+aP+\beta\Lambda+\int\eta\Xi\dd\nu-\mathfrak F
=\mathsf M\mathsf N^{-1}\mathsf M^\transpose$.
Thus the residual is $2\chi^t\mathsf M(u+\Theta X_-)$. The terminal condition follows from $P_T=G$.
Collecting terms in the stationarity expression gives
\begin{align}
&B^\transpose p^u_-+D^\transpose q^u
+\int_{\Zmark}F(e)^\transpose\varrho^u(e)\nu(\dd e)\notag\\
&\qquad+2\chi^t(SX_-+Ru)\notag\\
&\qquad=2\chi^t\big(\mathsf M^\transpose X_-+\mathsf N u\big)
=2\chi^t\mathsf N\big(u+\Theta X_-\big).
\label{eq:stationarity-residual}
\end{align}
For $u=u^*=-\Theta X^*_-$, both residuals vanish, proving \eqref{eq:physical-adjoint} and \eqref{eq:physical-stationarity}.

Linearity gives $X^{t,\xi;u^*+\varepsilon v}=X^*+\varepsilon X_1$. The quadratic cost expansion and the preceding moment bounds give the $L^1(\Pp)$ derivative
\begin{equation}\label{eq:first-variation-cost}
\begin{aligned}
&DJ^{\mathrm r}(t,\xi;u^*)[v]\\
&=2\E_t^{\Pp}\Bigg[\chi^t(T)\ip{GX^*(T)}{X_1(T)}\\
&\quad+\int_t^T\chi^t(s)\Big\{\ip{QX^*_-+S^\transpose u^*}{X_{1,-}}\\
&\hspace{8em}+\ip{SX^*_-+Ru^*}{v}\Big\}\dd s\Bigg].
\end{aligned}
\end{equation}
Finally, apply the product formula to $\ip{p}{X_1}$ and use \eqref{eq:physical-adjoint} and the state equation for $X_1$; the state-variation terms cancel, giving $DJ^{\mathrm r}(t,\xi;u^*)[v]=0$ by \eqref{eq:physical-stationarity}.
\end{proof}

\end{document}